\documentclass[11pt,reqno]{amsart}

\usepackage{amsthm}
\usepackage{amssymb}
\usepackage{latexsym}
\usepackage{multicol}
\usepackage{verbatim,enumerate}
\usepackage[usenames]{color}
\usepackage[colorlinks=true, linkcolor=blue, citecolor=blue, urlcolor=blue]{hyperref}
\usepackage{hyperref}
\usepackage{amsmath, amscd}
\usepackage{mathdots}
\usepackage{cite}
\usepackage{amsmath}

\newmuskip\pFqmuskip

\newcommand*\pFq[6][8]{%
  \begingroup 
  \pFqmuskip=#1mu\relax
  \mathcode`\,=\string"8000
  \begingroup\lccode`\~=`\,
  \lowercase{\endgroup\let~}\pFqcomma
  {}_{#2}F_{#3}{\left[\genfrac..{0pt}{}{#4}{#5};{#6}\right]}%
  \endgroup
}
\newcommand{\pFqcomma}{\mskip\pFqmuskip}

\newmuskip\pPqmuskip

\newcommand*\pPq[6][8]{%
  \begingroup 
  \pPqmuskip=#1mu\relax
  \mathcode`\,=\string"8000
  \begingroup\lccode`\~=`\,
  \lowercase{\endgroup\let~}\pPqcomma
  {}_{#2}\Phi_{#3}{\left[\genfrac..{0pt}{}{#4}{#5};{#6}\right]}%
  \endgroup
}
\newcommand{\pPqcomma}{\mskip\pPqmuskip}

\advance\textwidth by 1.3in \advance\oddsidemargin by -1.2in \advance\evensidemargin by -1.2in

\def\H{\mathfrak h} 
\def\CG{\mathfrak{Cg}} 
\def\CN{\mathfrak{Cn}} 
\def\CH{\mathfrak{Ch}} 

\def\d{\delta}

\theoremstyle{definition}
\newtheorem*{cor}{Corollary}
\newtheorem*{lem}{Lemma}
\newtheorem*{prop}{Proposition}

\newtheorem{thm}{Theorem}

\newtheorem*{defn}{Definition}
\newtheorem*{rem}{Remark}
\newtheorem*{thm*}{Theorem}

\theoremstyle{remark}

\newcounter{cnt}
\newenvironment{enumerit}{\begin{list}{{\hfill\rm(\roman{cnt})\hfill}}{%
\settowidth{\labelwidth}{{\rm(iv)}}\leftmargin=\labelwidth%
\advance\leftmargin by \labelsep\rightmargin=0pt\usecounter{cnt}}}{\end{list}} \makeatletter
\def\mydggeometry{\makeatletter\dg@YGRID=1\dg@XGRID=20\unitlength=0.003pt\makeatother}
\makeatother \theoremstyle{remark}

\numberwithin{equation}{section}

\newcommand{\nc}{\newcommand}
\newcommand{\rnc}{\renewcommand}

\newcommand{\qbinom}[2]{\genfrac[]{0pt}0{#1}{#2}}
\nc{\cal}{\mathcal} \nc{\goth}{\mathfrak} \rnc{\bold}{\mathbf}

\nc\btau{{\mbox{\boldmath $\tau$}}}
  \nc\bxi{{\mbox{\boldmath $\xi$}}}
\nc\bmu{{\mbox{\boldmath $\mu$}}} \nc\bcN{{\mbox{\boldmath $\cal{N}$}}} \nc\bcm{{\mbox{\boldmath $\cal{M}$}}} \nc\blambda{{\mbox{\boldmath
$\lambda$}}}\nc\bnu{{\mbox{\boldmath $\nu$}}}

\newcommand{\lie}[1]{\mathfrak{#1}}

\makeatletter
\def\section{\def\@secnumfont{\mdseries}\@startsection{section}{1}%
  \z@{.7\linespacing\@plus\linespacing}{.5\linespacing}%
  {\normalfont\scshape\centering}}
\def\subsection{\def\@secnumfont{\bfseries}\@startsection{subsection}{2}%
  {\parindent}{.5\linespacing\@plus.7\linespacing}{-.5em}%
  {\normalfont\bfseries}}
\makeatother

 \nc{\Hom}{\operatorname{Hom}}
  \nc{\mode}{\operatorname{mod}}
\nc{\End}{\operatorname{End}} \nc{\wh}[1]{\widehat{#1}} \nc{\Ext}{\operatorname{Ext}} \nc{\ch}{\text{ch}} \nc{\ev}{\operatorname{ev}}
\nc{\Ob}{\operatorname{Ob}} \nc{\soc}{\operatorname{soc}} \nc{\rad}{\operatorname{rad}} 

\nc{\N}{{\bold N}}  \nc\boa{\bold a} \nc\bob{\bold b} \nc\boc{\bold c} \nc\bod{\bold d} \nc\boe{\bold e} \nc\bof{\bold f} \nc\bog{\bold g}
\nc\boh{\bold h} \nc\boi{\bold i} \nc\boj{\bold j} \nc\bok{\bold k} \nc\bol{\bold l} \nc\bom{\bold m} \nc\bon{\mathbb n} \nc\boo{\bold o}
\nc\bop{\bold p} \nc\boq{\bold q} \nc\bor{\bold r} \nc\bos{\bold s} \nc\boT{\bold t} \nc\boF{\bold F} \nc\bou{\bold u} \nc\bov{\bold v}
\nc\bow{\bold w} \nc\boz{\bold z} \nc\boy{\bold y} \nc\ba{\bold A} \nc\bb{\bold B} \nc\bc{\mathbb C} \nc\bd{\bold D} \nc\be{\bold E} \nc\bg{\bold
G} \nc\bh{\bold H} \nc\bi{\bold I} \nc\bj{\bold J} \nc\bk{\bold K} \nc\bl{\bold L} \nc\bm{\bold M} \nc\bn{\mathbb N} \nc\bo{\bold O} \nc\bp{\bold
P} \nc\bq{\bold Q} \nc\br{\bold R} \nc\bs{\bold S} \nc\bt{\bold T} \nc\bu{\bold U} \nc\bv{\bold V} \nc\bw{\bold W} \nc\bz{\mathbb Z} \nc\bx{\bold
x}

\nc\chara{\operatorname{Char}}

\begin{document}

\title{Demazure flags, $q$--Fibonacci polynomials and hypergeometric series }

\author{Rekha Biswal}
\thanks{}
\address{The Institute of Mathematical Sciences, Chennai, India}
\email{rekha@imsc.res.in}

\author{Vyjayanthi Chari}
\thanks{}
\address{Department of Mathematics, University of California, Riverside, CA 92521}
\email{chari@math.ucr.edu}
\thanks{V.C was partially supported by DMS 1303052.}

\author{Deniz Kus}
\address{Mathematisches Institut, Universit\" at Bonn, Germany}
\email{dkus@math.uni-bonn.de}
\thanks{D.K was partially funded  under  the  Institutional  Strategy  of  the  University of Cologne within the German Excellence Initiative}

\subjclass[2010]{}
\begin{abstract}
We study a family of finite--dimensional representations of the hyperspecial parabolic subalgebra of the twisted affine Lie algebra of type $\tt A_2^{(2)}$. We prove that these modules admit a decreasing filtration whose sections are isomorphic to stable Demazure modules in an integrable highest weight module of sufficiently large level. In particular, we show that any stable level $m'$ Demazure module admits a filtration by level $m$ Demazure modules for all $m\ge m'$. We define the graded and weighted generating functions which encode the multiplicity of a given Demazure module and establish a recursive formulae. In the case when $m'=1,2$ and $m=2,3$ we determine these generating functions completely and show that they define hypergeoemetric series and that they are related to the $q$--Fibonacci polynomials defined by Carlitz.
\end{abstract}

\maketitle

\section*{Introduction}\label{section1}

The study of Demazure modules in highest weight representations of Kac--Moody algebras has been of interest for a long time.  A character formula for these modules,  analogous to the Weyl character formula,  was given in \cite{D74,KU02,M88}. Combinatorial versions of the character of such modules were given in \cite{Lit95}.  In  the case of affine Lie algebras there is extensive literature in the case of level one highest weight integrable modules; here the level is the integer  by which the canonical central element of the affine Lie algebra acts on the highest weight module. 
 The work of Sanderson \cite{San00} for $\tt A_n^{(1)}$ and the work of Ion \cite{I03} more generally,  shows that 
the  character of a particular family of level one Demazure modules (which we shall refer to as stable Demazure modules)   is given by a specialization of  Macdonad polynomials (in the untwisted simply--laced case) and by the specialization of the Koornwinder polynomial for the twisted affine Lie algebras.

In \cite{Jos06}, A. Joseph introduced the notion of a module admitting a Demazure flag. He proved in the case of the quantized enveloping algebra associated to a simply--laced affine Lie algebra that the tensor product of a one--dimensional Demazure module by an arbitrary Demazure module admits a filtration whose successive quotients are isomorphic to Demazure modules. It was shown in \cite{Na11} that an analogous result could be deduced from \cite{Jos06} for stable Demazure modules in the simply--laced untwisted affine Lie algebras.

In this paper we turn our attention to such questions in the case of twisted affine Lie algebras;  the most interesting situation being the Lie algebra  of type $\tt A_{2n}^{(2)}$ and we consider the corresponding rank one situation. The stable Demazure modules that we shall be interested in are those which admit an action of the hyperspecial parabolic subalgebra (denoted $\lie C\lie g$) of the affine Lie algebra.  Our first result constructs a large family (which includes the Demazure modules)  of finite--dimensional modules for $\lie C\lie g$ which admit a Demazure flag. Analogous results for $\tt A_1^{(1)}$ were established in \cite{CSSW14} using results from \cite{CV13}. In the current situation we use results from \cite{KV14}; however we have to work much harder to establish the analogs of the results of \cite{CSSW14} for two reasons. We  have to contend with the fact that $\tt A_2^{(2)}$ is a much more complicated algebra and we also have to prove additional representation theoretic results which were not established in \cite{KV14}.

The primary goal of this paper is to study the relationship between the theory of Demazure flags and  Koornwinder polynomials, Ramanujan mock--theta functions, Carlitz $q$--Fibonacci polynomials and more generally the theory of hypergeometric series.  These connections arise as follows; the algebra $\lie C\lie g$ is integer graded and the  Demazure modules admit a compatible grading. The corresponding generating series of graded multiplicities, along with certain weighted versions of these are what provide the link between the modules and number theory and combinatorics.

 As a first example, the non--symmetric Koornwinder polynomial $E_{-n}(q^2,t)$ at $t=\infty$ coincides with the graded character of $D(1,n)$ (for details see \cite{I03}). Our results show that  we can express the specialized Koornwinder polynomial as a {\em $\mathbb{N}[q]$--linear combination of graded characters of level $m$ Demazure modules  for any fixed $m\geq 1$}. The generating series of the graded multiplicities of the trivial module in a level three Demazure flag has interesting specializations; one of which gives rise to a fifth order mock theta function of Ramanujan. Analogous connections were made in the untwisted case in \cite{BCSV15}.

In this paper we introduce the weighted generating series of the multiplicities.   Namely we define the weighted multiplicity of a Demazure module $D$ occuring in a Demazure flag of a module $V$ by multiplying the graded multiplicity by a power of $q$ so that the resulting polynomial is either zero or has a non--zero constant term. We show that in the case of level two flags in level one Demazure modules the resulting weighted generating series is a specialization of the hypergeometric function $\pFq{1}{1}{a}{b}{q, z}$.  In the case of level three flags in level two Demazure modules, the generating series are determined explicitly and is essentially given by the $q$--Fibonacci polynomials defined by Carlitz. We remind the reader that the original $q$--analogs of the Fibonacci polynomials were introduced by Schur \cite{S73} in his work on the Rogers--Ramanujan identities; the latter identities are well--known to be related to the representation theory of affine Lie 
algebras, but in a very different context \cite{LW84}. Moreover, we give a closed form for the generating series for the numerical multiplicities ($q=1$) and find that they involve the Chebyshev polynomials of the second kind. Again analogous results along these lines were first proved in \cite{BCSV15}.

The case of higher level Demazure flags is much more complicated; however our results on weighted multiplicities do suggest that the hypergeometric series again appear, but this is still conjectural.
{\em The results of \cite{BCSV15}, the current paper and the work in progress  \cite{BCSW16}, clearly indicate a deep and unexpected connection between the theory of Demazure flags and its combinatorics  and number theory. }

The paper is organized as follows. In Section \ref{section3} we state the main results of the paper with the minimum possible notation. The representation theoretic results are established in  Sections \ref{section4} and Section \ref{section5}. The last sections are devoted to using the representation theory to calculate the graded and weighted multiplicities.

\textit{Acknowledgment: The third author thanks George Andrews and Volker Genz for many helpful discussions.}


\section{Preliminaries} \label{section2}

\subsection{} We denote the set of complex numbers by $\bc$ and, respectively, the set of integers, non--negative integers, and positive integers  by $\bz$, $\bz_+$, and $\bn$. We set $\bold N=\{(r,s): r,s\in\frac{1}{2}\mathbb{N}, r+s\in \bn\}$ and let $y_+=\max\{0,y\}$ for $y\in \mathbb{R}$. All vector spaces considered in this paper are $\bc$--vector spaces. For a $\mathbb{Z}$--graded vector space $V=\bigoplus_{k\in\mathbb{Z}}V[k]$ we denote by $\tau_p^{*}V$ the graded vector space whose $k$--th graded piece is $V[k+p]$. Given a complex Lie algebra $\mathfrak{a}$, we let $\bu(\mathfrak{a})$ be the corresponding universal enveloping algebra.

\subsection{}
We refer to \cite{K90} for the general theory of affine Lie algebras. The focus of this paper is the twisted affine Lie algebra $\widehat{\lie g}$ of type  $\tt A_{2}^{(2)}$, which contains the simple Lie algebra $\lie g=\mathfrak{sl}_2$ as a subalgebra. Recall that $\mathfrak{sl}_2$ is the complex simple Lie algebra of two by two matrices of trace zero and that $\{x_{0}, y_0, h_{0}\}$ is the standard basis with $[h_{0},x_{0}]= 2 x_{0}$, $[h_{0},y_{0}]=-2 y_{0}$ and $[x_{0},y_{0}]=h_{0}$. The element $h_{0}$ generates a Cartan subalgebra $\lie h$ of $\lie g$ and let $R=\{\pm\alpha\}$ be the set of roots with respect to $\lie h$. We fix $\widehat{\lie h}$ a Cartan subalgebra of $\widehat{\lie g}$ containing $\lie h$ and let $\widehat{R}$ the set of roots of $\widehat{\lie g}$ with respect to $\widehat{\lie h}$. The corresponding sets of positive and negative roots are denoted as usual by $\widehat R^\pm$ and $R^\pm$ respectively. If $\d$ denotes the unique non--divisible positive imaginary root in $\widehat{R}$, then we have $\widehat{R}=\widehat{R}^+\cup\widehat{R}^-$, where $\widehat{R}^-=-\widehat{R}^+$, $\widehat R^+ =\widehat R^+_{\rm {re}}\cup \widehat R^+_{\rm{im}}$, $\widehat R^+_{\rm{im}} =\mathbb N\delta$, and
$$\widehat{R}^+_{\rm{re}}=R^+\cup\big(R+2\mathbb N\delta\big)\cup\frac12\big(R+(2\mathbb Z_++1)\delta\big).$$
We also consider the set 
$$\widehat R_{\rm re}(\pm)=R^\pm\cup\big(R^\pm+2\mathbb N\delta\big)\cup\frac12\big(R^\pm+(2\mathbb Z_++1)\delta\big).$$
Given $\beta\in \widehat R$ let $\widehat{\lie g}_\beta\subset\widehat{\lie g} $ be the corresponding root space; note that $x_{0}$ (resp. $y_0$) is a generator of the root space $\widehat{\lie g}_{\alpha}$ (resp. $\widehat{\lie g}_{-\alpha}$). For any real root $\beta$ we fix a generator $x_{\beta}$ of $\widehat{\lie g}_\beta$ and abbreviate 
$$x_{\alpha+2r\delta}:=x_{2r},\quad x_{\frac{\alpha}{2}+(r+\frac{1}{2})\delta}:=x_{r+\frac{1}{2}},\quad x_{-\alpha+2r\delta}:=y_{2r},\quad x_{-\frac{\alpha}{2}+(r+\frac{1}{2})\delta}:=y_{r+\frac{1}{2}}.$$
\subsection{} We  define several subalgebras of $\widehat{\lie g}$ that will be needed in the rest of the paper. Let $\widehat{ \lie b}$ be the Borel subalgebra corresponding to $\widehat{R}^+$, and let $\widehat{\lie n}^+$ be its nilpotent radical,
 $$\widehat{\lie b}=\widehat{\lie h}\oplus \widehat{\lie n}^+,\quad \widehat{\lie n}^\pm =\bigoplus_{\beta\in\widehat R^+}\widehat{\lie g}_{\pm \beta}.$$ The subalgebras $\lie b$ and $\lie n^\pm$ of $\lie g$ are defined similarly. 
The twisted current algebra $\CG$ is defined as
 $$
 \CG=\lie h\oplus\widehat{\lie n}^+\oplus\lie n^-
 $$
and admits a triangular decomposition $$\CG=\lie C\lie n^+\oplus \CH \oplus\lie C\lie n^-,$$ where  
$$\CH=\CH_+\oplus \H, \quad \CH_+=\bigoplus_{k>0}\widehat{\lie g}_{k\delta},\quad \lie C\lie n^\pm=\bigoplus_{\beta\in \widehat{R}_{\rm re}(\pm)}\widehat{\lie g}_{\pm\beta}.$$ Following \cite{CIK14} we call $\CG$ the \textit{hyperspecial} twisted current algebra, which is different from the notion of twisted current algebras of type $\tt A_{2}^{(2)}$ that exists in the literature. The differences are clarified in \cite[Remark 2.5]{CIK14}. To simplify notation we set 
$$\mathbf{U}(\CG):=\mathbf{U},\quad \mathbf{U}(\lie C\lie n^\pm):=\mathbf{U}^{\pm}.$$
\subsection{}
The scaling operator $d\in \widehat{\lie h}$ defines a $\mathbb Z_+$--graded Lie algebra structure on $\CG$: for $\beta\in \widehat R$ we say that $\widehat{\lie g}_\beta$ has grade $k$ if $\beta(d)=k$. Since $\delta(d)=2$ the eigenvalues of $d$ are all integers and if $\lie g_\beta\subset\CG$, then the eigenvalues are non--negative integers.  With respect to this grading, the zero homogeneous component of the twisted current algebra is $\CG[0]= \lie g$.
A finite--dimensional $\mathbb{Z}_+$--graded $\CG$--module is a $\mathbb{Z}$--graded vector space admitting a compatible graded action of $\CG$:
$$V=\bigoplus_{k\in\mathbb{Z}}V[k],\quad \CG[r]V[k]\subset V[k+r].$$ Note that each graded component $V[k]$ is a $\mathfrak{g}$--module and we define the graded character as 
$$\text{ch}_{\text{gr}} V=\sum_{k\in \mathbb{Z}} \text{ch}_{\lie g} V[k] q^k.$$


\section{The main results} \label{section3}
We summarize the main results of the paper. We keep the notation to a minimum and refer the reader to the later sections for precise definitions.

\subsection{}\label{demdefi}
Given $m\in\mathbb{N}$ and $n\in\mathbb{Z}_+$ with $n=n_1m+n_0$, where $n_0,n_1\in \mathbb Z$,  $0<n_0\le m$, 
let $D(m,n)$ be the graded $\CG$--module generated by an element $v_n$ with  defining relations:
\begin{align}
&\label{weylreel}(\CN^+\oplus \CH_+)v_n=0,\ \ \   h_{0}v_n=n v_n, \ \ \  y_{0}^{n+1}v_n=0 &\\&
y_{2n_1+2}v_n=0,\ \ \ y_{n_1+\frac{3}{2}}v_n=0, \ \ \text{ if $m>1$}&\\&
y_{2n_1}^{n_0+1}v_n=0, \ \  \ y_{n_1+\frac{1}{2}}^{(2n_0-m)_++1}v_n=0,\ \text{ if $n_0<m$.}
\end{align}
It was proved in \cite{KV14} that $D(m,n)$ is a finite--dimensional indecomposable $\CG$--module and is isomorphic to a  Demazure module occurring in a highest weight integrable irreducible representation of $\widehat{\mathfrak{g}}$. We call $m$ the level of the Demazure module. 

\subsection{}  We define the notion of a Demazure flag as follows. Let $V$ be a graded  finite--dimensional $\CG$--module $V$; we say that $V$ admits  a a level $m$--Demazure flag if   there exists a decreasing sequence of graded $\CG$--submodules of $V$ 
$$\mathcal{F}(V)=\big(0\subset V_0\subset V_1\subset \cdots \subset V_k=V\big)$$
such that the successive quotients of the flag are isomorphic to $\tau_p^{*}D(m,n)$ for some $n,p\geq 0$.

Let $[V:\tau_p^{*}D(m,n)]$  be the number of times $\tau_p^{*}D(m,n)$  occurs as  a section of this flag. It is not hard to see that this number is independent of the choice of the flag. For an indeterminate $q$, we define a polynomial in $\mathbb{N}[q]$ by
$$[V:D(m,n)]_q=\sum_{p\geq 0} [V:\tau_p^{*}D(m,n)]q^p. $$ 
We also set
\begin{equation}\label{dies0}[V: D(m,n)]_q:=0, \mbox{ if $n<0$}.\end{equation}

We call $[V:D(m,n)]_q$ the graded multiplicity of $D(m,n)$ in $V$. 
If $[V:D(m,n)]_q$  is non--zero we define the {\em weighted  multiplicity of $D(m,n)$ in $V$}  to be the unique polynomial $[V: D(m,n)]_q^w$ in $\mathbb{N}[q]$  with non--zero constant term satisfying \begin{equation}\label{weightmutl}q^r [V: D(m,n)]_q^w= [V:D(m,n)]_q \ \ {\rm{for \ some}}\ \ r\in\bz_+.\ \end{equation} Otherwise we set $[V:D(m,n)]_q^w=[V: D(m,n)]_q=0$.

\subsection{} The first result of our  paper is the following:
\begin{thm}\label{mainsecthm1}
For all integers $m\geq m' >0$ and $s\geq 0$ the module $D(m',s)$ admits a Demazure flag of level $m$.
Moreover, 
\begin{equation}\label{dies1}[D(m',s): D(m,n)]_q=\delta_{n,s}, \mbox{ if $n\geq s$},\quad [D(m,s): D(m,n)]_q=\delta_{n,s}\end{equation} and for $m \ge \ell \ge m'>0$ we have
\begin{equation}\label{difflevel}
[D(m',s): D(m,n)]_q=\sum_{p\geq 0}  [D(m',s): D(\ell, p)]_q\,[D(\ell,p): D(m,n)]_q.
\end{equation}
\end{thm}

\begin{rem} In the case of quantized enveloping algebras associated with simply--laced
Kac--Moody Lie algebras the existence of such a flag was proved in \cite{Jos06} using the theory of canonical bases. Later, it was shown in \cite{Na11} that taking the classical limit, the result remains true for the corresponding affine Lie algebras.

An alternate constructive proof was given in \cite{CSSW14} in the case of $\tt A_1^{(1)}$; this proof enables one to compute multiplicities in the Demazure flag.  We follow this approach    in the current paper; however there are many non--trivial representation theoretic results that have to first be established for $\lie C\lie g$. In particular, we shall prove  in Section~\ref{section5}  that a more general family of modules admit a Demazure flag.
\end{rem}

\subsection{} Our next results deal with understanding the graded and weighted multiplicities of a level $m$ Demazure flag in $D(m',n)$ in the case when $(m',m)\in\{1,2),\ (2,3)\}$.
 Given $n \in \mathbb{Z}_{+}$ and $m \in \mathbb{Z}$, the $q$--binomial coefficient is defined by
$$\qbinom{n}{m}_q=\frac{(q;q)_n}{(q;q)_{n-m}(q;q)_m},\ \ n\geq m >0,
\ \ \qbinom{n}{0}_q=1, \ \ \qbinom{n}{m}_q=0, \ \ m <0 \text{ or } m>n,$$ 
where the $q$--Pochammer symbol $(a;q)_n$ is defined as
$$(a;q)_n=\prod_{i=1}^n (1-aq^{i-1}),\ n>0,\  (a;q)_0=1.$$
Recall the follwing well--known $q$--binomial identities:
$$\qbinom{n}{m}_q = \qbinom{n-1}{m}_q +q^{n-m}\qbinom{n-1}{m-1}_q,\quad \qbinom{n}{m}_q = q^{m}\qbinom{n-1}{m}_q +\qbinom{n-1}{m-1}_q.$$

For $s\in\bz_+$ let ${\text{res}}_2(s)\in\{0,1\}$ be defined by requiring $s-{\text{res}}_2(s)$ be even.
	The proof of the next proposition can be found in Section \ref{section6}.
\begin{prop}\label{1to2and2to3}\mbox{}  Let $s, p\in\bz_+$.
\begin{enumerit}
\item[(i)] We have, 
	\begin{equation}\label{1to2w}[D(1,s+p):D(2,s)]^w_q= \qbinom{\lfloor \frac{s}{2} \rfloor+p}{p}_{q^2},\end{equation} and \begin{equation}\label{1to2}D(1,s+p):D(2,s)]_q=q^{p(s+p+{\text{res}_2(s)})} [D(1,s+p):D(2,s)]^w_q.\end{equation}
\item[(ii)]  For $0\le r\le 5$,  let $$r'=\delta_{r,1}+\delta_{r,4},\ \ \  \bar{r}=(\delta_{r,1}+\delta_{r,3}+\delta_{r,5})\text{res}_2(p)-\delta_{r,1},\ \ \tilde r=\lfloor \frac{r}{3} \rfloor.$$ Setting $n=6s+r$, we have
$$[D(2,n+p): D(3,n)]_q^w=\sum_{j=0}^{\lfloor{\frac p2}\rfloor} q^{2j(j+r'+\text{res}_2(p))}\qbinom{2s+\tilde r+ \lfloor\frac p2\rfloor -j}{2s+\tilde r}_{q^2}\qbinom{s+j+\bar r}{2j+\text{res}_2(p)}_{q^2},$$

and \begin{equation}\label{2to3}[D(2,n+p): D(3,n)]_q=q^{p(4s+r-\tilde r+\lceil \frac{p}{2}\rceil)+\text{res}_2(p)(r'+\tilde r-\lceil \frac{p}{2}\rceil)} [D(2,n+p): D(3,n)]_q^w.\end{equation}
\end{enumerit}
\end{prop}

\begin{rem} One outcome of our results is the following.  It was  proved in \cite{I03} that the specialization of the non--symmetric Koornwinder polynomial $E_{-n}(q^2,t)$ at $t=\infty$ coincides with the graded character of $D(1,n)$.  Using  Theorem~\ref{mainsecthm1} with $m'=1$, we see that we can express 
$E_{-n}(q^2,t)$ as a $\mathbb{N}[q]$--linear combination of graded characters of level $m$ Demazure modules. In the case when $m'=1$ and $m=2,3$ our analyses gives closed formulae for this decomposition of Koornwinder polynomials. 
\end{rem}

\subsection{}  Given $n\in\mathbb{Z}$, $m',m\in\mathbb{N}$ with $n\geq 0$ and $m\geq m'$ we define  generating series which encode the graded  and weighted  multiplicities of a level $m$ flag in a level $m'$ Demazure module:
$$A_n^{m' \rightarrow m}(x,q):= \sum_{p\ge 0}[D(m',n+p):D(m,n)]_{q}x^p,\ \ \  A_n^{ m' \rightarrow m, w}(x,q):=\sum_{p\ge 0}[D(m',n+p):D(m,n)]^w_{q}x^p.$$
We shall relate these series to general basic hypergeometric series defined by
$$\pFq{r}{s}{a_1,\dots,a_r}{b_1,\dots,b_s}{q,z}=\sum_{n\geq 0} \frac{(a_1;q)_n(a_2;q)_n\cdots (a_r;q)_n}{(b_1;q)_n(b_2;q)_n\cdots (b_s;q)_n}\frac{z^n}{(q;q)_n}.$$
For more details and properties of hypergeometric series we refer the reader to \cite{SL66}.

\subsubsection{} Consider  the case when $(m', m)=(1,2)$.  Proposition~\ref{1to2and2to3}(i) gives that $$[D(1,2n+p):D(2, 2n)]_q^w= [D(1, 2n+1+p): D(2,2n+1)]_q^w,$$ and hence we set
$$\Phi_{n+1}^{1\to 2}(x,q):=A_{2n+1}^{1 \rightarrow 2, w}(x,q)= A_{2n}^{1 \rightarrow 2, w}(x,q).$$ A further application of Proposition~\ref{1to2and2to3}(i) gives $$\Phi_{n+1}^{1\to 2}(x,q) =\sum_{p\geq 0}\qbinom{n +p}{n}_{q^2}x^p=\sum_{j\geq 0}\qbinom{j}{n}_{q^2}x^{j-n}.$$
Using the identity \begin{equation}\label{ident}\sum_{j\ge 0}\qbinom{j}{k}_qx^j= \dfrac{x^k}{(x:q)_{k+1}},\end{equation}
we get $$\Phi_{n+1}^{1 \rightarrow 2}(x,q)=\dfrac{1}{(x;q^2)_{n+1}}.$$ It follows that if we set $\Phi^{1\to 2}_{0}=1$, then $$\sum_{n\ge 0} \Phi^{1\to 2}_{n}(x,q)z^{n}= \pFq{1}{1}{q^2}{x}{q^2,z}.$$ 

\subsubsection{} We now consider the case when $m=3$.
 In  the case when $(m', m)=(1,3)$ we  prove the following,
\begin{prop}
We have
$$A_0^{1 \rightarrow 3}(1,q)=\phi_0(q),\quad qA_1^{1 \rightarrow 3}(1,q)=\phi_1(q),$$ where $$\phi_0(q)=\sum_{n\ge 0}q^{n^2}(-q;q^2)_n,\qquad \phi_1(q)=\sum_{n\ge 0}q^{(n+1)^2}(-q;q^2)_n,$$ are the fifth order mock--theta functions of Ramanujan.\end{prop}
\proof Using equation \eqref{difflevel}  we see that 
$$A_0^{1 \rightarrow 3}(x,q)=\sum_{p\geq 0}[D(1,p): D(3,0)]_qx^p=\sum_{p,s\geq 0}[D(1,p): D(2,s)]_q[D(2,s): D(3,0)]_qx^p.$$
Equation \eqref{2to3} gives,  $$[D(2, 2j+1): D(3,0)]_q=0,\qquad [D(2, 2j): D(3,0)]=q^{2j^2},$$ and hence 
using  Proposition~\ref{1to2and2to3}(i)  we get 
\begin{align*}A_0^{1 \rightarrow 3}(x,q)&=\sum_{p,j\geq 0}q^{2j^2+(p-2j)p}\qbinom{p-j}{p-2j}_{q^2}x^p&\\&
=\sum_{i,j\geq 0}q^{j^2+i^2}\qbinom{i}{j}_{q^2}x^{i+j}&\\&
=\sum_{i\geq 0}q^{i^2}x^i\sum_{j\geq 0}q^{j^2-j}\qbinom{i}{j}_{q^2}(qx)^{j}&\\&
=\sum_{i\geq 0}q^{i^2}(-qx;q^2)_ix^i.
\end{align*}
Thus, $A_0^{1 \rightarrow 3}(1,q)=\phi_0(q)$. The proof in the other case is similar.
\endproof

\subsubsection{} We now consider the case $(m',m)=(2,3)$. In this case the generating series $A_{n}^{2 \rightarrow 3, w}(x,q)$ is related to the $q$--Fibonacci polynomials defined by Carlitz in \cite{Car75}: 
$$S_{n}(x,q)_0=x S_{n-1}(x,q)_0+q^{n-2}S_{n-2}(x,q)_0,\ \ \ S_0(x,q)_0=0,\ \  S_1(x,q)_0=1.$$
We remark that the specialization $S_{n}(1,q)_0$ was first considered by Schur \cite{S73} in his proof of the Rogers--Ramanujan identities; see also \cite{And04} for more details.

 The solution to this recurrence is $$S_{n+1}(x,q)_0=\sum_{j\ge 0}\qbinom{n-j}{j}_qq^{j^2}x^{n-2j}.$$
The same recurrence relation but  with different initial conditions, $$S_{n}(x,q)_1=xS_{n-1}(x,q)_1+q^{n-2}S_{n-2}(x,q)_1,\ \ \ S_{-1}(x,q)_1=0,\ \ S_0(x,q)_1=1,$$  has the solution $$S_n(x,q)_1=\sum_{j\ge 0}\qbinom{n-j}{j}_qq^{j(j-1)}x^{n-2j}.$$
Write $$A_{n}^{2 \rightarrow 3, w}(x,q)=A_{n}^{2 \rightarrow 3, w}(x,q)_0+A_{n}^{2 \rightarrow 3, w}(x,q)_1,$$
where 
$$A_n^{2 \rightarrow 3, w}(x,q)_k:=\sum_{p\ge 0}[D(2,n+2p+k):D(3,n)]^w_{q}x^{2p+k},\ \ k\in \{0,1\}.$$
We prove,
\begin{prop}
For $0\le r\le 5$ we set
$$s_0=s- \delta_{r,1},\ \ s_1= s-1+\delta_{r,3}+\delta_{r,5}.$$ Then,

$$A_{6s+r}^{2 \rightarrow 3, w}(x,q)_k=\frac{q^{-2s_k^2-k}}{(x^2;q^2)_{2s+\lfloor \frac{r}{3} \rfloor}}\ S_{2s_k+1}(y,q^2)_k,\ \ k\in \{0,1\},$$
where $y=q^{2s_k+r'}x$.
\end{prop}
\begin{proof}
Using the formulae in Proposition~\ref{1to2and2to3}(ii) we get
  \begin{align*}
A_{6s+r}^{2 \rightarrow 3, w}(x,q)_0&
=\sum_{p,j\ge 0}q^{2j(j+r')}\qbinom{2s+\tilde r+p-j}{2s+\tilde r}_{q^2}\qbinom{s+j+\bar r}{2j}_{q^2}x^{2p}
&\\&=\sum_{i\geq 0}\qbinom{2s+\tilde r+i}{i}_{q^2}x^{2i}\left(\sum_{j\geq0} q^{2j(j+r')}\qbinom{s_0+j}{2j}_{q^2}x^{2j}\right)&\\&=\frac {1}{(x^2;q^2)_{2s+\lfloor \frac{r}{3} \rfloor}}\left(\sum_{j\geq0} q^{2j(j+r')}\qbinom{s_0+j}{2j}_{q^2}x^{2j}\right)&\\&
=\frac{q^{-2s_0^2}}{(x^2;q^2)_{2s+\lfloor \frac{r}{3} \rfloor}}\ S_{2s_0+1}(y,q^2)_0.\end{align*}
The remaining case works similarly. \end{proof}

\subsubsection{}
The generating series can also be viewed as limits of hypergeometric series and we thank George Andrews for helping us with this observation. In particular,
\begin{align*}(x^2;q^2)_{\lfloor \frac{n}{3} \rfloor+1}&A_{6s+r}^{2 \rightarrow 3, w}(x,q)_0&\\&
=\sum_{j\ge 0}q^{2j(j+r')}\frac{(-1)^jq^{2s_0j-j^2+j}(q^{2s_0+2};q^2)_j(q^{-2s_0};q^2)_j}{(q^2;q^2)_{2j}}x^{2j}&\\&
=\sum_{j\ge 0}q^{2j(s_0+r')+j(j+1)}(-1)^j\frac{(q^{2s_0+2};q^2)_j(q^{-2s_0};q^2)_j}{(q^2;q^2)_j(-q^2;q^2)_j(q;q^2)_j(-q;q^2)_j}x^{2j}&\\&=\lim_{t\to0}\ \pFq{4}{3}{q/t,t,q^{2s_0+2},q^{-2s_0}}{-q^2,q,-q}{q^2,tx^2q^{2(s_0+r')+1}}.\end{align*}

\subsection{} Our final collection of results discuss the generating functions for the numerical multiplicities, namely  we set $$A_n^{m' \rightarrow m}(x):=\sum_{p\geq 0}[D(m',n+p):D(m,n)]_{q=1}x^p,$$ and study these function for $m'=1$ and also when $m=m'+1$. In both cases they are rational functions in $x$; moreover in the first case they are related to the Chebyshev polynomial  of the second kind as we now discuss. For $n\in\bz_+$ define  polynomials $a_n(x)$ by  $a_0(x)=a_1(x)=1$ and for $n\geq 2$ 
 \begin{equation}\label{defa_n}
a_n(x) =
\begin{cases}
a_{n-1}(x)-xa_{n-2}(x) & \text{ if } n \text{ is odd},\\
(1+x)a_{n-1}(x)-xa_{n-2}(x)  & \text{ if } n \text{ is even}.\\
\end{cases}
\end{equation} We shall prove, \begin{thm}\label{thmgenser1} For  $s\in\bz_+$ and $r\in\{0,\dots,m-1\}$  we have 
$$A_{ms+r}^{1 \rightarrow m}(x)=\dfrac{a_{2m-2r-1}(x)}{(a_m(x) a_{m+1}(x))^{s+1}},$$ if $\lfloor\frac m2\rfloor \leq r  \leq m-1$ and $$A_{ms+r}^{1 \rightarrow m}(x)=\dfrac{a_m(x) a_{m- 2r-1}(x)}{(a_m(x) a_{m+1}(x))^{s+1}},$$ if $ 0 \leq r \leq \lfloor \frac{m}{2} \rfloor -1$.

\end{thm}
\begin{rem} The connection with Chebyshev polynomials is made as follows. Consider the following recurrences: $$U_{n+1}(x)=2xU_n(x)-U_{n-1}(x), \ \ U_0(x)=1, \ \ U_1(x)=2x.$$   Let $P_n(x)$ be the polynomials defined by the recurrence 
$$P_0(x)=P_1(x)=1,\ P_{n+1}(x)=P_n(x)-xP_{n-1}(x) \text{ for } n \geq 1.$$ It is known that the Chebyshev polynomials of the second kind satisfy the recurrences for $U$ and that $P_n(x^2)=x^n U_n((2x)^{-1})$.
It is not hard to see that the polynomials $a_n(x)$ are given by
$$a_n(x)=(1+x)^{\lfloor \frac{n}{2}\rfloor }P_n\left(\dfrac{x}{1+x}\right),\ n\geq 0.$$ Basically one just checks that the right hand side of the preceding equation satisfies the same recurrence relations as the $a_n(x)$. It is also useful to note here that
$a_{2n}(x)=a_{2n-1}(x)-x^2 a_{2n-3}(x)$ and that 
$a_n(x)=(1-x)a_{n-2}(x)-x^2 a_{n-4}(x)$ for $n\geq 4$.
\end{rem}



 
\section{The modules \texorpdfstring{$V(\bxi)$}{V} and dimension bounds} \label{section4}
In this section we state the  more general  version of Theorem \ref{mainsecthm1}.

\subsection{}
Let $\cal P_\ell$ be the set of all partitions $\bxi$ of length $\ell+1$ such that the following holds: \begin{equation}\label{partform}\bxi=(\xi_0\geq\underbrace{(\xi+1)\ge \cdots\ge(\xi+1)}_{\ell-1-p}\ge\underbrace{\xi\ge \cdots\ge\xi}_{p}\ge \xi_{\ell}> 0)\end{equation} where either $p=0$ and $\ell=1$ or $1\le p\le \ell-1$ if $\ell>1$.  For $\bxi=(\xi_0\ge\xi_1\ge\cdots \ge \xi_\ell)\in\cal P_\ell$, set  $$|\bxi|=\sum_{j\ge 1}\xi_j,\qquad 
\phi(\bxi;k):=\begin{cases}
\sum_{j\geq k+1} \xi_{j}\ -\ \frac{1}{2}\xi_{k+1},& \text{ if $0\leq k\leq \ell-2$ }\\
\big(\xi_{\ell}-\frac{1}{2}\xi_{\ell-1}\big)_{+},& \text{ if $k=\ell-1$}\\
\ \  0 & \text{ else.}\end{cases}
$$ 
Define a partial order on $\cal P=\cup_{\ell\in\bn}\cal P_\ell$ by:  for $\bxi_j\in\cal P_{\ell_j}$, $j=1,2$, we say that 
 \begin{equation}\label{ord}\bxi_1 \prec \bxi_2  \iff  {\rm{either}} \ \ell_1<\ell_2\ \ {\rm{or}} \ \ \ell_1=\ell_2\ \ {\rm{and}} \ \  \bxi_1 < \bxi_2,\end{equation} where $<$ denotes the usual  reverse lexicographic order on partitions.

\subsection{} We introduce the main objects of this paper.
\begin{defn}\label{third} Given $\bxi\in\cal P_\ell$ with $|\bxi|=n$, 
we define $V(\bxi)$ to be the graded quotient of $D(1,n)$ by the submodule generated by the graded elements,
\begin{align}& y_{2\ell}v_{n},\  \ \  y_{\ell+\frac{1}{2}}v_{n},&\\&
\label{finer2}
 y_{2\ell-2}^{(\xi_{\ell}+1)}v_{n},\ \ {\rm{if}}\ \  \xi_{\ell}+1<\xi_{\ell-1} ,\ \ \  y_{\ell-\frac{1}{2}}^{(2\phi(\bxi;\ell-1)+1)}v_{n},\ \ {\rm{if}}\ \ \ \xi_{\ell}<\xi_{\ell-1}.
\end{align}\end{defn}
 If  $n_1\in\mathbb{Z}$ with $n_1\geq -1$ and $n=n_1m+n_0,\ 0<n_0\leq m$ then we have an isomorphism of graded $\CG$--modules
\begin{equation}\label{specialcase}V(\bxi)\cong D(m,n),\ \  \  \bxi=(\underbrace{m\ge \cdots\ge m}_{n_1+1}\ge n_0).\end{equation}
\begin{rem}\label{remh}
Note that the definition of $V(\bxi)$ is independent of $\xi_0$ unless $\ell=1$ and $2\xi_{1}>\xi_{0}$. In what follows we denote by $v_{\bxi}$ the cyclic generator of $V(\bxi)$.
\end{rem}

\subsection{} We now state the more general version of Theorem \ref{mainsecthm1}; the proof of this theorem can be found in the next section. 
\begin{thm}\label{mainthm1} For  $\bxi\in\mathcal{P}_{\ell}$, we have an isomorphism of $\lie g$--modules $$V(\bxi)\cong D\big(\xi_1,\xi_1\big)^{\otimes (\ell-1-p)}\otimes D(\xi_{\ell-1},\xi_{\ell-1})^{\otimes p}\otimes D(\xi_{\ell-1},\xi_{\ell}).$$
Further, the  $\CG$--module $V(\bxi)$ admits a level $m$--Demazure flag if and only if $m\geq \xi_0$.
\end{thm} 
 \begin{rem}\label{remsim}\mbox{}
 \begin{enumerate}
 \item More generally the isomorphism  in Theorem~\ref{mainthm1} is  of $\CG$--modules if we replace the tensor product by the fusion product; we refer the reader to \cite{FL99} for the definition and properties of fusion products.
Hence our result gives a presentation of a certain class of fusion products of different level Demazure modules. 
\item If $V(\bxi)$ admits a level $m$ Demazure flag, then we have $m\geq \xi_0$. This implication can be proven similarly as in \cite[Lemma 3.7]{CSSW14} and the details will be omitted. 
\end{enumerate}
\end{rem}

\subsection{}  The modules $V(\bxi)$, $\bxi\in\cal P_\ell$ admit another realization which we now discuss.  We shall need some  notation.
For $s,r\in\bn$, $k\in\mathbb{Z}_+$ let 
\begin{equation*}\label{maxs}\bs_{\geq k}(r,s)=\Big\{\bold b=(b_j)_{\substack{j\in\bz_+ \\ j\ge k}}:  b_j\in\bz_+,\ \   \sum_{j\ge k} b_j=r,\ \ \ \ \sum_{j\ge k} jb_j=s\Big\}, \end{equation*}
and for $(r,s)\in\bold N$, $\tilde k\in\mathbb{Z}_+/2$ let $$\widetilde{\bs}_{\geq \tilde k}(r,s) =\Big\{\widetilde\bob=( \tilde b_j)_{\substack{2j\in\bz_+ \\ j\ge \tilde k}} :\tilde  b_j\in\bz_+ ,\ \   \sum_{j\ge \tilde k}(\tilde b_{j+1/2}+2\tilde b_j)=2r,\ \ \sum_{j\ge \tilde k}\big((j+\frac12)\tilde b_{j+\frac12}+2j\tilde b_j\big)=s\Big\}.$$
The following elementary calculation will be used repeatedly.

\begin{lem}\label{lamm}

Let $(r,s)\in\bold N$, and $\widetilde \bob\in \widetilde{\bs}_{\geq k-\frac{3}{2}}(r,s)$ (resp. $\widetilde \bob\in \widetilde{\bs}_{\geq k-\frac{5}{2}}(r,s)$) such that $\tilde b_j=0$ for all $j\geq k$. Then we have
$$\tilde b_{k-\frac{1}{2}}+\tilde b_{k-1}=s-2r(k-3/2),\quad \tilde b_{k-\frac{3}{2}}+\tilde b_{k-1}=2r(k-1/2)-s$$
\big(resp. $\tilde b_{k-\frac{3}{2}}+2\tilde b_{k-\frac{1}{2}}+\tilde b_{k-2}+3\tilde b_{k-1}=s-2r(k-5/2),\quad \tilde b_{k-\frac{5}{2}}-\tilde b_{k-\frac{1}{2}}+\tilde b_{k-2}-\tilde b_{k-1}=2r(k-3/2)-s$\big).
\end{lem}\qed
\subsection{} 
For any non--negative integer $b\in\bz_+$ and $x\in \CG$ set $x^{(b)}:=x^b/b!$. Let  $\bold y_{\geq k}(r,s),\ \widetilde{\bold y}_{\geq\tilde  k}(r,s)$ be the following elements of $\bu$:

\begin{align}\label{xlm}&\bold y_{\ge k}(r,s)=\sum_{\bold b\in\bs_{\geq k}(r,s)}y_{2k}^{(b_k)}\ y_{2k+2}^{(b_{k+1})}\cdots y_{2s}^{(b_s)},&\\&
\label{ydefn}\widetilde{\bold y}_{\geq \tilde k}(r,s)=\sum_{\bold{\widetilde{\bob}}\in \widetilde{\bs}_{\geq \tilde  k}(r,s)} \overset{\rightarrow}{\prod}_{n\geq \tilde  k}
\  (\widetilde{y}_{n+\frac12})^{(\tilde b_{n+\frac{1}{2}})}\ 
(\widetilde{y}_{2n})^{{(\tilde b_n)}},
\end{align}
where for $n\in\bz_+$ we set 
$$2^n\widetilde{y}_{n+\frac12}:=-y_{n+\frac12}\qquad 
2^{2n}\widetilde{y}_{2n}:=
\left((-1)^n-2\right)
y_{2n},
$$
and $\overset{\rightarrow}{\prod}_{n\geq 0}$ refers to the product of the specified factors written exactly in the increasing order of the indexing parameter. The following is proved along the same lines as Theorem 1 and Theorem 7 of \cite{KV14}.
\begin{prop}\label{genmax2}  For $\bxi\in\cal P_\ell$ with $|\bxi|=n$, the module $V(\bxi)$ is the quotient of $D(1,n)$ by the additional relations:
\begin{align}\label{third1}&\bold y_{\geq k}(r,s)v_n=0,\ \forall  s,r\in\bn,\ k\in \bz_+ \mbox{ with } \ \ s+r\ge 1+kr+\sum_{j\ge k+1}\xi_j,\\
\label{third2} &\widetilde{\bold y}_{\ge k+\frac{1}{2}}(r,s)v_n=0, \ \forall (r,s)\in\bold N, \ k\in\bz_+\mbox{ with} \ \ s\ge \frac{1}{2}+2kr+\phi(\bxi;k).\end{align} Moreover, there exists a surjective map of $\lie g$--modules  $$V(\bxi)\twoheadrightarrow D\big(\xi_1,\xi_1\big)^{\otimes (\ell-1-p)}\otimes D(\xi_{\ell-1},\xi_{\ell-1})^{\otimes p}\otimes D(\xi_{\ell-1},\xi_{\ell}).$$
\hfill\qedsymbol
\end{prop}

It is immediate from the proposition  that
\begin{equation}\label{cthird}y_{2k}^{(r+1)}v_\bxi=0,\quad \big(\mbox{resp. } y_{k+\frac{1}{2}}^{(r+1)}v_\bxi=0\big)\end{equation}
for all $r,k\in\bz_+$ with $r\ge \sum_{j\ge k+1}\xi_j$ \big(resp. $r\ge 2\phi(\bxi;k)$\big).


\section{Demazure flags and recursive formulae} \label{section5}

In this section, we prove Theorem~\ref{mainthm1} by an induction on $\ell$.
\subsection{}\label{remdim}
 Recall from \cite{KV14} that the Demazure module $D(\xi_0,\xi_{1})$ is irreducible if and only if $2\xi_{1}\leq \xi_0$ and otherwise decomposes into irreducible finite--dimensional $\mathfrak{g}$--modules as follows
$$D(\xi_0,\xi_{1})\cong \tau^{*}_0V(\xi_1)\oplus \cdots \oplus \tau^{*}_{2\xi_1-\xi_0}V(\xi_0-\xi_{1}).$$
Hence $$\dim D(\xi_0,\xi_{1})= (\xi_1+1)+\frac{1}{2}\big((2\xi_1-\xi_0)_+(\xi_0+1)\big).$$

Thus Proposition~\ref{genmax2}  gives  a lower bound for the dimension of $V(\bxi)$, $\bxi\in\mathcal{P}_{\ell}$:
\begin{equation}\label{dimest}\dim V(\bxi)\geq\binom{\xi_1+2}{2}^{\ell-1-p}\binom{\xi_{\ell-1}+2}{2}^{p}\Big((\xi_{\ell}+1)+\frac{1}{2}\big((2\xi_1-\xi_{\ell-1})_+(\xi_{\ell-1}+1)\big)\Big).\end{equation}

\subsection{}\label{inistep}
To see that induction begins at $\ell=1$ we write  $\bxi=(\xi_0,\xi_1)$.  Equation \eqref{specialcase} implies that  $V(\bxi)\cong D(\xi_0,\xi_1)$ and hence the first statement of Theorem~\ref{mainthm1} holds in this case. If $2\xi_1\leq \xi_0$ or $\xi_0=m$ the second statement also holds since $D(\xi_0,\xi_1)\cong D(m,\xi_1)$. Otherwise consider the filtration of graded $\mathbf{U}$--modules
\begin{align}\label{filt}
0\subset \mathbf{U}y^{2\xi_1-\xi_0}_{\frac{1}{2}}v_{\bxi}\subset \mathbf{U}y^{2\xi_1-\xi_0-1}_{\frac{1}{2}}v_{\bxi}\subset\cdots\subset\mathbf{U}y^{2\xi_1-\xi_0-s}_{\frac{1}{2}}v_{\bxi}\subset V(\bxi),
\end{align}
where $s=(2\xi_1-\xi_0) - (2\xi_1-m)_+-1$. 
Using Definition~\ref{third} it is easily seen that the successive quotients of the filtration  in  \eqref{filt} are themselves  quotients of a Demazure module of the form $\tau^{*}_s D(m,\xi_1-s)$.  The dimension inequality in \eqref{dimest} now implies  these maps are isomorphisms and hence $V(\bxi)$ has a Demazure flag of level $m$ establishing the inductive step. Moreover we have also proved that 
\begin{equation}\label{anfan}[D(\xi_0,\xi_1),D(m,s)]_q=\begin{cases} q^{\xi_1-s},& \text{ if $s=\xi_1 \mbox{ or } (2\xi_1-m)_{+}< \xi_1-s \leq (2\xi_1-\xi_0)_{+}$}\\
0,& \text{else.}\end{cases}\end{equation}

\subsection{}\label{ideas} \textit{For the rest of this section we fix an arbitrary $\bxi\in\mathcal{P}_{\ell}$ with $\ell>1$ and assume that the main theorem holds for all $\btau\in \mathcal{P}$ with $\btau \prec\bxi$. We further suppose $\xi_0=\xi_1$ (see Remark~\ref{remh}).} 
\vskip 12pt

Define $\bxi^+\in\cal P$ by:
\begin{equation*}\bxi^+:=((\xi+1)^{\ell-p+1},\xi^{p-1},\xi_{\ell}-1).\end{equation*}
We have $\phi(\bxi^+; \ell-1)=(\xi_{\ell}-\frac12 \xi_{\ell-1}-1-\delta_{p,1}/2)_+$.
Noting that $\bxi^+\prec\bxi$, we have by the induction hypothesis 
that $V(\bxi^+)$ admits a level $m$ Demazure flag iff $m\geq \xi+1$. We now prove,
\begin{prop}\label{propker}  The assignment $v_\bxi\mapsto v_{\bxi^+}$  defines  a surjective morphism  $\varphi^+:V(\bxi)\to V(\bxi^+)\to 0$ of $\CG$--modules. Moreover $\ker\varphi^+$ is generated by the elements,
$$y_{2\ell-2}^{(\xi_\ell)}v_\bxi,\ \ \ \delta_{p,2}\delta_{\xi_\ell,1}y_{\ell-\frac32}^{(\xi_{\ell-1})}v_{\bxi} ,\ \ \  {\rm{if}}\ \ 2\xi_{\ell}\leq \xi_{\ell-1}$$
$$y_{\ell-\frac{3}{2}}v_{\bxi},\ \  {\rm{if}}\ \ \ p=2,\ \  \xi_\ell=1=\xi_{\ell-1},$$ and in all other cases 
by
$y_{\ell-\frac{1}{2}}^{((2\xi^+_\ell-\xi^+_{\ell-1})_++1)}v_{\bxi}.$ 

\begin{proof} A simple checking shows that $v_{\bxi^+}$ satisfies all the relations that $v_\bxi$ does and hence the existence of $\varphi^+$ is immediate. It is also immediate that
 $\ker\varphi^+$ is generated by the elements 
\begin{align}\label{propker1}
y_{2\ell-2}^{(\xi_\ell)}v_\bxi,\ \ \  y_{\ell-\frac{1}{2}}^{((2\xi^+_{\ell}-\xi_{\ell-1}^+)_++1)}v_{\bxi},\ \ 
\delta_{p,2}\delta_{\xi_\ell,1}y_{\ell-\frac{3}{2}}^{(\xi_{\ell-1})}v_{\bxi}.
\end{align}
If $2\xi_\ell\le \xi_{\ell-1}$ the result follows since  $y_{\ell-\frac12}v_\bxi=0$. 
If    $2\xi_\ell>\xi_{\ell-1}$, we   prove by a downward induction on $k$ that for all $0\leq k\leq (2\xi^+_{\ell}-\xi_{\ell-1}^+)_+$, we have  \begin{equation}\label{hilfst}y_{2\ell-2}^{(\xi_\ell-k)} y_{\ell-\frac{1}{2}}^{(k)}v_\bxi\in\bu  y_{\ell-\frac{1}{2}}^{((2\xi^+_{\ell}-\xi_{\ell-1}^+)_++1 )}v_{\bxi}.\end{equation}
Assuming we have done this, notice that by taking $k=0$ the proof is complete unless we are in the case of $p=2$ and $\xi_\ell=\xi_{\ell-1}=1$ when we also have to prove that $$y_{\ell-\frac12}v_\bxi\in\bu y_{\ell-\frac32}v_\bxi.$$But this is immediate by applying an element of $\CH_+$ of appropriate grade. We set $2r=2\xi_{\ell}-k$ and $s=2(\ell-1)\xi_\ell-(\ell-\frac{3}{2})k$. Note that
\begin{equation}\label{hhg}s\geq \frac{1}{2}+(2\xi_\ell-k)(\ell-2)+\frac{1}{2}\xi_{\ell-1}+\xi_\ell=\frac{1}{2}+2r(\ell-2)+\phi(\bxi,\ell-2).\end{equation} 
For an arbitrary element $\widetilde\bob\in \widetilde{\bs}_{\geq \ell-\frac{3}{2}}(r,s)$ with $\tilde b_j=0$ for $j\geq \ell$ we get with Lemma~\ref{lamm}
$$\tilde b_{\ell-\frac{1}{2}}+\tilde b_{\ell-1}=\xi_{\ell},\quad \tilde b_{\ell-\frac{3}{2}}+\tilde b_{\ell-1}=\xi_\ell-k,$$
which implies $\tilde b_{\ell-\frac{3}{2}}=\tilde b_{\ell-\frac{1}{2}}-k\geq 0$.
Hence, together with \eqref{third2} and \eqref{hhg} we get 
\begin{align}\label{us3345}y_{2\ell-2}^{(\xi_\ell-k)}\ y_{\ell-\frac{1}{2}}^{(k)}v_\bxi\in\sum_{t>k} \mathbf{U} y_{2\ell-2}^{(\xi_\ell-t)}\ y_{\ell-\frac{1}{2}}^{(t)}v_\bxi.\end{align}
If $k=(2\xi^+_{\ell}-\xi_{\ell-1}^+)_+$ equation \eqref{hilfst} is immediate with \eqref{us3345}. Otherwise we know by the induction hypothesis that each summand in \eqref{us3345} has the desired property.
\end{proof}
\end{prop}

In the rest of this section we shall show by doing a case by case analysis that $\ker\varphi^+$ has a filtration such that the successive quotients are of the form $V(\bxi')$, with  $\bxi'\prec\bxi$.  Since the induction hypothesis applies to the $\bxi'$ it follows that $\ker\varphi^+$ has a level $m$ Demazure flag and also allows us to get an upper bound for $V(\bxi)$; together with the lower bound established in Proposition \ref{genmax2} we then complete the inductive step.

\subsection{} Given  $\bxi=( (\xi+1)^{\ell-p},\xi^p,\xi_\ell)\in\mathcal{P}_\ell$, set  $$k(\bxi)=\begin{cases}
2,& \text{ if $2\xi_\ell-\xi_{\ell-1}\geq 3$ and $p=1$}\\
1,& \text{ if $2\xi_\ell-\xi_{\ell-1}=2$ or $2\xi_\ell-\xi_{\ell-1}\geq 3$ and $p>1$}\\
0,& \text{ if $2\xi_\ell-\xi_{\ell-1}=1$}\\
-1,& \text{ if $2\xi_\ell-\xi_{\ell-1}\leq 0$.}
\end{cases}$$ Equivalently, 
\begin{equation}\label{defnkxi}k(\bxi)=(2\xi_{\ell}-\xi_{\ell-1})_+-(2\xi^+_{\ell}-\xi^+_{\ell-1})_+-1.\end{equation}

For $ -1\leq j\leq k(\bxi)$, define  partitions $\bxi(j)$ as follows:
\begin{equation*}\bxi(j)=((\xi+1)^{\ell-p+1-\delta_{j,0}-\delta_{j,-1}},\xi^{p-1+\delta_{j,0}},\xi_{\ell-1}-\xi_\ell+\delta_{j,2}).\end{equation*} It is easily seen that $\bxi(j)\in\cal P$ and $\bxi(j)\prec\bxi$ for $-1\le j\le k(\bxi)$.

\subsection{} We analyze $\ker\varphi^+$ under the assumption that $2\xi_{\ell}>\xi_{\ell-1}$.
\begin{prop}\label{case1} Suppose that  $2\xi_{\ell}>\xi_{\ell-1}$. 
\begin{enumerit} 
\item[(i)] Let  $\delta_{\xi_{\ell},\xi_{\ell-1}}\delta_{p,2}=0$.  For  $-1\le j\le k(\bxi)$ there exists graded $\CG$--submodules $V_j\subset \ker\varphi^+ $  with \begin{gather*}V_{-1}\cong \delta_{\xi_\ell\xi_{\ell-1}}\delta_{p,1}\
\tau^{*}_{4\xi_{\ell}(\ell-1)}V(\bxi(-1)),\ \   \ \  V_{k(\bxi)}=\ker\varphi^+, \\ 
V_{j-1}\subset V_{j },\ \ \ \ V_j/V_{j-1}\cong \tau^*_{(2\ell-1)(2\xi_\ell-\xi_{\ell-1}-j)}V(\bxi(j)),
\ \ 0\le j\le k(\bxi).
 \end{gather*}
\item[(ii)]   Let  $\delta_{\xi_{\ell},\xi_{\ell-1}}\delta_{p,2}=1$. If $\xi_\ell=1$ 
we have a short exact sequence of graded $\CG$--modules
$$0\rightarrow \tau^*_{(2\ell-1)}V(\bxi(0))\rightarrow \ker\varphi_+ \rightarrow  \tau^*_{(2\ell-3)}V((\xi_{\ell-1}+1)^{\ell-1})\rightarrow 0.$$
If $\xi_\ell>1$ then there exists graded $\CG$--submodules $V_0\subset V_1$ of $\ker\varphi^+$ such that \begin{gather*}V_0\cong \tau^*_{\xi_{\ell}(2\ell-1)}V(\bxi(0)),\qquad V_1/V_0\cong \tau^*_{4\xi_{\ell}(\ell-1)-(2\ell-1)}V((\xi_{\ell-1}+1)^{\ell-1}),\\ \ker\varphi^+/V_1\cong \tau^*_{(\xi_{\ell}-1)(2\ell-1)}V(\bxi(1)).\end{gather*}

\end{enumerit}
\end{prop}

\subsection{} To complete our analysis of $\ker\varphi^+$ we shall prove,
\begin{prop}\label{case2} Assume that $2\xi_\ell\le \xi_{\ell-1}$. 
\begin{enumerit}
\item[(i)] Suppose that $\delta_{\xi_\ell, 1}\delta_{p,2}=0$ and $ (1-\delta_{2\xi_\ell,\xi_{\ell-1}})\delta_{p,1}=0$. Then we have an isomorphism of graded $\CG$--modules
$$\tau^*_{4(\ell-1)\xi_{\ell}}V(\bxi(-1))\cong \ker \varphi^+.$$
\item[(ii)] Suppose that $\delta_{\xi_\ell, 1}\delta_{p,2}=0$ and $ (1-\delta_{2\xi_\ell,\xi_{\ell-1}})\delta_{p,1}=1$. We have a short exact sequence of graded $\CG$--modules $$0\rightarrow \tau^{*}_{(2\ell-3)\xi_{\ell-1}+2\xi_{\ell}}V((\xi_{\ell-1}+1)^{\ell-1},\xi_{\ell})\rightarrow\ker \varphi^+\rightarrow \tau^{*}_{4(\ell-1)\xi_{\ell}}V(\bxi(-1))\rightarrow 0.$$
\item[(iii)] Suppose that $\delta_{\xi_\ell, 1}\delta_{p,2}=1$. We have a short exact sequence
of graded $\CG$--modules
$$0\to \tau_{4(\ell-1)}^*V(\bxi(-1))\to\ker\varphi^+\to \tau_{\xi_{\ell-1}(2\ell-3)}^*V((\xi_{\ell-1}+1)^{\ell-1})\to 0.$$

\end{enumerit}
\end{prop}
\subsection{}
Before we prove Proposition~\ref{case1} and Proposition~\ref{case2} we show how we can use both to prove Theorem~\ref{mainthm1}. We obtain a filtration of $\ker\varphi^+$ by graded $\CG$--submodules such that the successive quotients are of the form $V(\bxi')$, with  $\bxi'\prec\bxi$. By applying the induction hypothesis to each $V(\bxi')$ and $V(\bxi^+)$ we obtain that $V(\bxi)$ admits a level $m$ Demazure flag for all $m\geq \xi+1$. Hence the second part of Theorem~\ref{mainthm1} is proven unless we are in the case $\xi_0=\xi=m$ when we also have to prove that $V(\bxi)$ has a level $\xi_0$ Demazure flag.
But this is immediate with \eqref{specialcase} since  $V(\bxi)$ itself is a Demazure module of level $\xi_0$.\par
Recall the dimension bound from \eqref{dimest}. Now we also have an upper bound using the above filtration. A long but tedious calculation shows that these bounds coincide and hence we have equality in \eqref{dimest}. This proves the first part of the theorem together with Proposition~\ref{genmax2}. Moreover, the explicit construction of the filtration yields the following recursive formulae:

\begin{cor}\label{c1}\mbox{}
Let $\bxi\in \mathcal{P}_{\ell}$, $m\in \mathbb{N}$ such that $m\geq \xi_0$ and $D$ a level $m$ Demazure module. 
\begin{enumerate}
\item
If $2\xi_{\ell}>\xi_{\ell-1}$ we have 
\begin{align*}[V(\bxi):D]_q&=[V(\bxi^+):D]_q
+\sum_{j=0}^{k(\bxi)}q^{(2\ell-1)(2\xi_\ell-\xi_{\ell-1}-j)}[V(\bxi(j)):D]_q&\\&+\delta_{\xi_\ell,\xi_{\ell-1}}(\delta_{p,1}+\delta_{p,2})q^{4\xi_\ell(\ell-1)-(p-1)(2\ell-1)}[V((\xi_{\ell-1}+1)^{\ell-1}):D]_q.
\end{align*}
\item If $2\xi_{\ell}\leq\xi_{\ell-1}$ we have 
\begin{align*}[V(\bxi):D]_q&=[V(\bxi^+):D]_q+q^{4\xi_\ell(\ell-1)}[V(\bxi(-1)):D]_q&\\&+q^{(2\ell-3)\xi_{\ell-1}+2\xi_\ell}(1-\delta_{2\xi_\ell,\xi_{\ell-1}})\delta_{p,1}[V((\xi_{\ell-1}+1)^{\ell-1},\xi_\ell):D]_q&\\&+q^{(2\ell-3)\xi_{\ell-1}}\delta_{\xi_\ell,1}\delta_{p,2}[V((\xi_{\ell-1}+1)^{\ell-1}):D]_q.
\end{align*}

\end{enumerate}
\qed
\end{cor}

\subsection{}
We need the  following technical lemma whose proof we postpone to the end of the section.
\begin{lem}\label{hilfslem}
We have the following relations in $V(\bxi)$:
\begin{enumerate}[(i)]
\item
$$y_{2\ell-4}^{(\xi_{\ell-1}-\xi_{\ell}+1)}\ y_{2\ell-2}^{(\xi_{\ell})}v_{\bxi}=0.$$

\item If  $2\xi_{\ell}\leq \xi_{\ell-1}$, we have 
$$y_{\ell-\frac{3}{2}}^{(\xi_{\ell-1}-2\xi_{\ell}+1)}\ y_{2\ell-2}^{(\xi_{\ell})}v_{\bxi}\ =\  0 \ = \ y_{2\ell-4}^{(\xi_{\ell}+1)}\
y_{\ell-\frac{3}{2}}^{(\xi_{\ell-1}-2\xi_{\ell})}\ y_{2\ell-2}^{(\xi_{\ell})}
v_{\bxi}$$
\item If $\xi_{\ell}= \xi_{\ell-1}$ we have 
$$y_{\ell-\frac{3}{2}}^{(\xi_{\ell}+1)}\ y_{\ell-\frac{1}{2}}^{(\xi_{\ell})}
v_{\bxi}\ =\  0,\quad \ y_{\ell-\frac{3}{2}}^{(\xi_{\ell})}\ y_{\ell-\frac{1}{2}}^{(\xi_{\ell})}
v_{\bxi}\in \mathbf{U}^- y_{2\ell-2}^{(\xi_{\ell})}v_{\bxi}. $$
\item  If $\xi_{\ell}=\xi_{\ell-1}=\xi_{\ell-2}$ we have 
$$\big\{y_{\ell-\frac{3}{2}}^{(\xi_{\ell}+1)}\ y_{\ell-\frac{1}{2}}^{(\xi_{\ell}-1)}
v_{\bxi},\ \   y_{2\ell-4}\ y_{\ell-\frac{3}{2}\d}^{(\xi_{\ell})}\ y_{\ell-\frac{1}{2}}^{(\xi_{\ell}-1)}
v_{\bxi}\big\}\subset \mathbf{U}^- y_{\ell-\frac{1}{2}}^{(\xi_{\ell})}v_{\bxi}.$$

\end{enumerate}
\end{lem}

\subsection{Proof of Proposition~\ref{case1}(i)}\label{fall11} 
Suppose that $2\xi_{\ell}>\xi_{\ell-1}$ and $\delta_{\xi_{\ell},\xi_{\ell-1}}\delta_{p,2}=0$. Set $$w_{-1}=\delta_{\xi_\ell\xi_{\ell-1}}\delta_{p,1}\  y_{\ell-\frac{3}{2}}^{(\xi_{\ell})}\ y_{\ell-\frac{1}{2}}^{(\xi_{\ell})}v_{\bxi},\qquad w_j=y_{\ell-\frac{1}{2}}^{(2\xi_{\ell}-\xi_{\ell-1}-j)}v_\bxi,\ \ 0\le j\le k(\bxi).$$ 
Let $V_j$ be the $\CG$--submodule of $V(\bxi)$ generated by  
$w_j$ for $-1\le j\le k(\bxi)$ and note that $V_{k(\bxi)}=\ker\varphi^+$ by Proposition~\ref{propker} and \eqref{defnkxi}. An easy calculation show that we have a surjective map from the appropriate level one Demazure module onto $V_j/V_{j-1}$, $0\leq j\leq k(\bxi)$. We will further show that the highest weight vector $\overline{w}_j$ in $V_j/V_{j-1}$ satisfies the defining relations of $V(\bxi(j))$. It means, we shall prove that
\begin{align}\label{finer01}&y_{2\ell}\overline{w}_j=y_{\ell+\frac{1}{2}}\overline{w}_j=y_{\ell-\frac{1}{2}}\overline{w}_j=0, &\\&
\label{finer02}
\ y_{2\ell-2}^{(\xi_{\ell-1}-\xi_{\ell}+\delta_{j,2}+1)}\overline{w}_j= \delta_{\xi_\ell,\xi_{\ell-1}}\delta_{p,1}\delta_{j,0}\ y_{\ell-\frac{3}{2}}^{(\xi_{\ell})}\overline{w}_j=0.
\end{align}
The relations in \eqref{finer01} are obviously satsified since
$[y_{2\ell}, y_{\ell-\frac{1}{2}}]=0,\ [y_{\ell+\frac{1}{2}}, y_{\ell-\frac{1}{2}}]=y_{2\ell},
$
and 
$$y_{\ell-\frac{1}{2}}^{(2\xi_{\ell}-\xi_{\ell-1}+1)}v_{\bxi}=y_{\ell+\frac{1}{2}}v_{\bxi}=y_{2\ell}v_{\bxi}=0$$
by \eqref{cthird}. Now we turn our attention to the relations in \eqref{finer02}. By construction we have 
$$\delta_{\xi_\ell,\xi_{\ell-1}}\delta_{p,1}\ y_{\ell-\frac{3}{2}}^{(\xi_{\ell})}w_0=w_{-1}=0.$$
Hence it remains to show
\begin{equation}\label{12}
y_{2\ell-2}^{(\xi_{\ell-1}-\xi_\ell+\delta_{j,2}+1)}\ w_j\in \mathbf{U} w_{j-1}.
\end{equation}
Set 
$$r=\frac{1}{2}\xi_{\ell-1}+1-\frac{\delta_{j,1}}{2}\ \mbox{ and }\ s=(\ell-\frac{1}{2})(2\xi_{\ell}-\xi_{\ell-1}-j)+2(\ell-1)(\xi_{\ell-1}-\xi_\ell+\delta_{j,2}+1).$$
Similarly as in the proof of Proposition~\ref{propker} we can show that each $\widetilde\bob\in \widetilde{\bs}_{\geq \ell-\frac{3}{2}}(r,s)$ with $\tilde b_i=0$ for $i\geq \ell$ satisfies $\tilde b_{\ell-\frac{1}{2}}\geq 2\xi_{\ell}-\xi_{\ell-1}-j$. Together with Proposition~\ref{genmax2} and
$$s\geq \frac{1}{2}+2r(\ell-2)+\phi(\bxi,\ell-2)$$
we get 
\begin{align*}
\widetilde{y}_{2\ell-2}^{(\xi_{\ell-1}-\xi_\ell+\delta_{j,2}+1)}\ & \widetilde{y}_{\ell-\frac{1}{2}}^{(2\xi_{\ell}-\xi_{\ell-1}-j)}v_{\bxi} \in\sum_{t>2\xi_{\ell}-\xi_{\ell-1}-j}\mathbf{U} \widetilde{y}_{\ell-\frac{1}{2}}^{(t)}v_{\bxi},
\end{align*}
which implies \eqref{12}. Thus we get surjective maps 
\begin{equation}\label{ss1}\tau^*_{(2\ell-1)(2\xi_\ell-\xi_{\ell-1}-j)}V(\bxi(j))\rightarrow V_j/V_{j-1},\ \ 0\leq j \leq k(\bxi).\end{equation}
Moreover, we note that  Lemma~\ref{hilfslem} (i) and (iii) imply that we also have a surjective map 
\begin{equation}\label{ss2}\delta_{\xi_\ell,\xi_{\ell-1}}\delta_{p,1}\ \tau^{*}_{4\xi_{\ell}(\ell-1)}V(\bxi(-1))\rightarrow V_{-1}.\end{equation}
The proof is complete if we show that the maps in \eqref{ss1} and \eqref{ss2} are isomorphisms. Applying our induction hypothesis, we obtain an upper bound for $\dim V(\bxi)$ by \eqref{ss1} and \eqref{ss2}. A straightforward calculation shows that the lower bound established in \eqref{dimest} coincides with this upper bound, which proves the proposition.

\subsection{Proof of Proposition~\ref{case1}(ii)}\label{fall12}
Assume that $2\xi_{\ell}>\xi_{\ell-1}$ and $\delta_{\xi_{\ell},\xi_{\ell-1}}\delta_{p,2}=1$.   Set 
$$w_0=y_{(\ell-\frac{1}{2})}^{(\xi_{\ell})}v_{\bxi},\ \  \ w_1=y_{\ell-\frac{3}{2}}^{(\xi_{\ell})}\ y_{\ell-\frac{1}{2}}^{(\xi_{\ell}-1)}v_{\bxi},\ \ w_2=(1-\delta_{\xi_\ell,1})y_{\ell-\frac{1}{2}}^{(\xi_{\ell}-1)}v_{\bxi},$$ and let $V_j$ be the submodule generated by the elements $w_i$, $0\le i\le j$. Note that $V_2=\ker\varphi^+$ by Proposition~\ref{propker} and if  $\xi_\ell=1$ we have $V_1=V_2=\ker\varphi^+$.\par
The idea is again to show that there is a cyclic generator satisfying the defining relations of $V(\bxi(j))$ and $V((\xi_{\ell-1}+1)^{\ell-1})$ respectively and to use the dimension bound given in \eqref{dimest}. The harder relations are stated in Lemma~\ref{hilfslem} and all other relations are easy to check. To be more precise, in addition to the relations proven in Lemma~\ref{hilfslem} (iii)--(iv) we need to verify 
$$y_{2\ell-2}\ y_{\ell-\frac{1}{2}}^{(\xi_{\ell}-1)}v_{\bxi}\in \mathbf{U}^- y_{\ell-\frac{1}{2}}^{(\xi_{\ell})}v_{\bxi},\quad y_{2\ell-2}\ y_{\ell-\frac{1}{2}}^{(\xi_{\ell})}v_{\bxi}=0.$$
The proof is similar and will be omitted.

\subsection{Proof of Proposition~\ref{case2}(i)}\label{fall22}
Assume that $2\xi_{\ell}\leq \xi_{\ell-1}$, $\delta_{p,2}\delta_{\xi_\ell,1}=0$ and $(1-\delta_{2\xi_{\ell},\xi_{\ell-1}})\delta_{p,1}=0.$ 
By Lemma~\ref{hilfslem} and \eqref{cthird} we know that $y_{2\ell-2}^{(\xi_{\ell})}v_{\bxi}$, which is the generator of the kernel by Proposition~\ref{propker}, satisfies the defining relations of $\tau^*_{4(\ell-1)\xi_{\ell}}V(\bxi(-1))$ given in Definition~\ref{third}. Again, a simple dimension argument using our induction hypothesis and \eqref{dimest} finishes the proof.
\subsection{Proof of Proposition~\ref{case2}(ii)}\label{fall23}We consider the case $2\xi_{\ell}\leq \xi_{\ell-1}$, $\delta_{p,2}\delta_{\xi_\ell,1}=0$ and $(1-\delta_{2\xi_{\ell},\xi_{\ell-1}})\delta_{p,1}=1$. Set
$$w_0=y_{\ell-\frac{3}{2}}^{(\xi_{\ell-1}-2\xi_{\ell})}y_{2\ell-2}^{(\xi_{\ell})}v_{\bxi},\ \  \ w_1=y_{2\ell-2}^{(\xi_{\ell})}v_{\bxi},$$ and let $V_j$ be the submodule generated by the element $w_j$, $0\leq j\leq 1.$ Note that Proposition~\ref{propker} implies $V_1=\ker \varphi^+$. From Lemma~\ref{hilfslem} and \eqref{cthird} we obtain the following surjective maps
\begin{equation}\label{zuzeigen0}\tau^{*}_{4(\ell-1)\xi_{\ell}}V(\bxi(-1))
\twoheadrightarrow V_1/V_0\end{equation}
and
\begin{equation}\label{zuzeigen}\tau^{*}_{(2\ell-3)\xi_{\ell-1}+2\xi_{\ell}}V((\xi_{\ell-1}+1)^{\ell-1},\xi_{\ell})\twoheadrightarrow V_0.\end{equation}
Again a simple dimension argument shows that the maps in \eqref{zuzeigen0} and \eqref{zuzeigen} are isomorphisms. 
\subsection{Proof of Proposition~\ref{case2}(iii)}\label{fall24} We consider the remaining case $2\xi_{\ell}\leq \xi_{\ell-1}$ and $\delta_{p,2}\delta_{\xi_\ell,1}=1$. Set
$$w_0=y_{2\ell-2}v_{\bxi},\ \  \ w_1=y_{\ell-\frac32}^{(\xi_{\ell-1})}v_{\bxi},$$ and let $V_j$ be the submodule generated by the element $w_i$, $0\leq i\leq j.$ Note that $V_1=\ker \varphi^+$.
Lemma~\ref{hilfslem} immediately implies 
$$\tau_{4(\ell-1)}^*V(\bxi(-1)) \twoheadrightarrow V_0.$$
Again by a dimension argument it remains to show that the highest weight vector in $V_1/V_0$ satisfies 
$$y_{2(\ell-2)}\ y_{\ell-\frac{3}{2}}^{(\xi_{\ell-1})}v_{\bxi}=y_{\ell-\frac{3}{2}}^{(\xi_{\ell-1}+1)}v_{\bxi}=0.$$
Let $k=(\ell-3),\ 2r=\xi_{\ell-1}+1$ and $s=(\ell-\frac{3}{2})(\xi_{\ell-1}+1)$. We obviously have $s\geq \frac{1}{2}+2r(\ell-3)+\phi(\bxi,\ell-3)$ and hence $\widetilde{\bold y}_{\geq \ell-\frac{5}{2}}(r,s)v_{\bxi}=0$. Our aim is to prove that 
\begin{equation}\label{zzu}\widetilde{\bold y}_{\geq \ell-\frac{5}{2}}(r,s)v_{\bxi}=y_{\ell-\frac{3}{2}}^{(\xi_{\ell-1}+1)}v_{\bxi}+V_0\end{equation}
or equivalently 
$$\widetilde \bob\in \widetilde{\mathbf{S}}_{\geq \ell-\frac{3}{2}}(r,s) \Longrightarrow \tilde b_{\ell-1}>0 \mbox{ or } \tilde b_{\ell-\frac{3}{2}}=\xi_{\ell-1}+1.$$
Assume $\tilde b_{\ell-1}=0$. We have with Lemma~\ref{lamm}
$$\tilde b_{\ell-\frac{3}{2}}+\tilde b_{\ell-2}=\xi_{\ell-1}+1,\quad \tilde b_{\ell-\frac{5}{2}}+\tilde b_{\ell-2}=0,$$ 
which proves \eqref{zzu}. Now we set $2r=\xi_{\ell-1}+2,\ s=(\ell-\frac{3}{2})\xi_{\ell-1}+2(\ell-2)$ and obtain with similar calculations as above that any $\widetilde \bob\in \widetilde{\mathbf{S}}_{\geq \ell-\frac{5}{2}}(r,s)$ with $\tilde b_{i}=0$ for $i\geq \ell-1$ is of the form 
$$\tilde b_{\ell-2}=1,\ \tilde b_{\ell-\frac{3}{2}}=\xi_{\ell-1},\ \tilde b_{\ell-1}=\tilde b_{\ell-\frac{5}{2}}=0\quad \mbox{or}\quad \tilde b_{\ell-2}=\tilde b_{\ell-1}=0,\ \tilde b_{\ell-\frac{3}{2}}=\xi_{\ell-1}+1,\ \tilde b_{\ell-\frac{5}{2}}=1.$$
Now \eqref{zzu} finishes the proof.

\subsection{Proof of Lemma~\ref{hilfslem}}
\textit{Proof of part (i): } Let
$r=1+\xi_{\ell-1}$, $s=\xi_{\ell}+(\ell-2)(\xi_{\ell-1}+1)$. We have 
$$r+s\geq 1+r(\ell-2)+\xi_{\ell-1}+\xi_{\ell}$$
which implies together with \eqref{third1}
$$\mathbf{y}_{\geq \ell-2}(r,s)v_{\bxi}=y_{2\ell-4}^{(\xi_{\ell-1}-\xi_{\ell}+1)}\ y_{2\ell-2}^{(\xi_{\ell})}v_{\bxi}=0.$$
\textit{Proof of part (ii): } We fix $t\in\{0,\dots,\xi_\ell\}$ and set $2r=2t+\xi_{\ell-1}+1$, $s=\xi_{\ell}+(\ell-\frac{3}{2})(\xi_{\ell-1}+1)+2(\ell-2)t$. We obtain
$$s\geq \frac{1}{2}+(2t+\xi_{\ell-1}+1)(\ell-2)+\frac{1}{2}\xi_{\ell-1}+\xi_{\ell}=\frac{1}{2}+2r(\ell-2)+\phi(\bxi,\ell-2)$$
and hence with \eqref{third2} we get
\begin{equation}\label{latne}\widetilde{\bold y}_{\ge \ell-\frac{3}{2}}(r,s)v_\bxi=\widetilde{y}_{\ell-\frac{3}{2}}^{(\xi_{\ell-1}-2\xi_{\ell}+4
t+1)}\ \widetilde{y}_{2\ell-2}^{(\xi_{\ell}-t)}v_{\bxi}=0,
\end{equation}
which proves the first part. In order to prove the second part we fix $z\in\{1,\dots,\xi_{\ell}+1\}$ and set $2r=\xi_{\ell-1}+2z$, $s=\xi_{\ell}+\xi_{\ell-1}(\ell-\frac{3}{2})+2(\ell-2)z$. If $\ell=2$ the statement is obvious; so let $\ell\geq 3$. We have
$$s\geq \frac{1}{2}+2r(\ell-3)+\phi(\bxi,\ell-3)$$
and therefore $\widetilde{\bold y}_{\geq \ell-{\frac{5}{2}}}(r,s)v_{\bxi}=0$. Let $\widetilde \bob\in \widetilde{\mathbf{S}}_{\geq \ell-{\frac{5}{2}}}(r,s)$ with $\tilde b_j=0$ for all $j\geq \ell-\frac{1}{2}$ and set $\tilde b_{\ell-1}=\xi_{\ell}-t$ for some $t\in\{0,\dots,\xi_{\ell}\}$. From Lemma~\ref{lamm} we get
\begin{equation}\label{kpp}
s-r(2\ell-5)=(\tilde b_{\ell-\frac{3}{2}}+3\tilde b_{\ell-1}+\tilde b_{\ell-2})=\xi_{\ell-1}+\xi_\ell+z\end{equation}
which implies  
$$z-\xi_\ell=2r-(\xi_{\ell-1}+\xi_\ell+z)=(\tilde b_{\ell-\frac{5}{2}}+\tilde b_{\ell-2}-\tilde b_{\ell-1}),\quad \tilde b_{\ell-\frac{5}{2}}+\tilde b_{\ell-2}=z-t.$$
Thus $\tilde b_{\ell-2}\leq z-t$ and with \eqref{kpp} we obtain
$$\tilde b_{\ell-\frac{3}{2}}\geq \xi_{\ell-1}-2\xi_{\ell}+4t.$$
Now we use \eqref{latne} and obtain that $\widetilde \bob$ must be of the form
$$\tilde b_{\ell-1}=\xi_{\ell}-t,\ \tilde b_{\ell-\frac{3}{2}}=\xi_{\ell-1}-2\xi_{\ell}+4t,\ \tilde b_{\ell-2}=z-t,\ \tilde b_{\ell-\frac{5}{2}}=0, \text{ for some $t\in\{0,\dots,\xi_\ell\}$}.$$
Therefore
\begin{equation}\label{summe3}\widetilde{\bold y}_{\geq \ell-\frac{5}{2}}(r,s)v_{\bxi}=\sum_{t=0}^{\min\{z,\xi_{\ell}\}}
\widetilde{y}_{\ell-\frac{3}{2}}^{(\xi_{\ell-1}-2\xi_{\ell}+4t)}\ \widetilde{y}_{2\ell-4}^{(z-t)}\ \widetilde{y}_{2\ell-2}^{(\xi_{\ell}-t)}v_{\bxi}=0.\end{equation}
Acting with $\widetilde{y}_{2\ell-4}^{(\xi_\ell+1-y)}$ on \eqref{summe3} yields 

\begin{equation}\label{summe2}\sum_{t=0}^{\min\{z,\xi_{\ell}\}}\frac{(\xi_{\ell}+1-t)!}{(z-t)!(\xi_{\ell}+1-z)!}\ 
\widetilde{y}_{\ell-\frac{3}{2}}^{(\xi_{\ell-1}-2\xi_{\ell}+4t)}\ \widetilde{y}_{2\ell-4}^{(\xi_{\ell}+1-t)}
\widetilde{y}_{2\ell-2}^{(\xi_{\ell}-t)}v_{\bxi}=0.\end{equation}
Our aim is to show that each summand in \eqref{summe2} vanishes. Let $v=(v_0,\dots,v_{\xi_{\ell}})$, where 
$$v_t=\widetilde{y}_{\ell-\frac{3}{2}}^{(\xi_{\ell-1}-2\xi_{\ell}+4t)}\ \widetilde{y}_{2\ell-4}^{(\xi_{\ell}+1-t)}\ \widetilde{y}_{2\ell-2}^{(\xi_{\ell}-t)}v_{\bxi},\quad 0\leq t \leq \xi_\ell.$$
By using \eqref{summe2} for $z=1,\dots,\xi_{\ell}+1$ we get a system of linear equations $Av=0$ where $A$ is invertible. Hence $v=0$. The proof of the remaining parts is similar and will be omitted.



\section{Proof of Proposition~\ref{1to2and2to3}}\label{section6}
In this section, we establish a recursive formulae for the graded multiplicities and prove Proposition~\ref{1to2and2to3}.
\subsection{}
The following lemma is the analogue of \cite[Lemma 3.8]{CSSW14} and is proven similarly using Corollary~\ref{c1}. The proof will be omitted.
\begin{lem}\label{propsec1}
Let $\bxi\in\mathcal{P}_{\ell}$, $\ell>1$, and $\overline\bxi$ be the unique partition obtained from $\bxi$ by removing $\xi_0$.
For $s\in\mathbb{Z}_+$, we have 
$$[V(\bxi):D(\xi_0,s)]_q=q^{2(|\bxi|-s)}[V(\overline\bxi): D(\xi_0, s-\xi_0)]_q.$$
\end{lem}
\qed
\subsection{}
Now we are able to give a recursive formulae for the graded multiplicities.
\begin{prop}\label{hilfslem1}Let $m,n\in\mathbb{N}$.
\begin{enumerate}
\item	We write $s=m s_1+s_0$, where $0 < s_0\leq m$ and set $\bxi=((m+1),m^{s_1},s_0)$. \\

\begin{enumerate}
\item[(i)] 
If $2s_0\leq m$ and $s_1 \geq 1$, then
\begin{align*}
\hspace{1cm}[D(m,s):D(m+1,n)]_q&=q^{2(s-n)}[D(m,s-m-1):D(m+1,n-m-1)]_q&\\&+q^{4s_1s_0}[D(m,s-2s_0):D(m+1,n)]_q.
\end{align*}
\item[(ii)] 
If $2s_0 > m$ and $s_1 \geq 1$, then 
\begin{align*}
\hspace{2,1cm}[D(m,s):D(&m+1,n)]_q=q^{2(s-n)}[D(m,s-m-1):D(m+1,n-m-1)]_q&\\&+q^{(2s_1+1)(2s_0-m)}[D(m,s+m-2s_0):D(m+1,n)]_q&\\&+\sum_{j=1}^{k(\bxi)}q^{2s_1(2s_0-j)+j+m-2n}[D(m,s-2s_0+j-1):D(m+1,n-m-1)]_q.
\end{align*}
\end{enumerate}
\item For $j \in \mathbb{Z}_{+}$, $0 \leq n, k \leq m$ and $m j+k \geq n$, 
\begin{equation*}
[D(m,m j+ k):D(m +1,n)]_q=
\begin{cases}
q^{j(m j + 2k)} &\text{ if } n=k, \\
q^{(j + 1)(m j+ 2k - m)} &\text{ if } n=m-k, \\
0 &\text{ otherwise. }
\end{cases}
\end{equation*}
\end{enumerate}

\begin{proof}
Using the recursive formulae given in Corollary~\ref{c1} and Lemma~\ref{propsec1} we get in part (i)
\begin{align*}
[D(m,s):D(m+1,n)]_q&=q^{2(s-n)}[D(m,s-m-1):D(m+1,n-m-1)]_q&\\&+q^{4s_1s_0}[D(m+\delta_{s_1,1},s-2s_0):D(m+1,n)]_q&\\&+\delta_{n,m+1}\delta_{s_0,1}\delta_{s_1,2}q^{3m}+(1-\delta_{2s_0,m})\delta_{s_0,n}\delta_{s_1,1}q^{m+2s_0}.
\end{align*}
The statement follows from the following identities, which are easy consequences of \eqref{anfan}:
$$[D(m,s-2s_0):D(m+1,n)]_q=[D(m+\delta_{s_1,1},s-2s_0):D(m+1,n)]_q+(1-\delta_{2s_0,m})\delta_{s_0,n}\delta_{s_1,1}q^{m-2n}$$
and
\begin{align*}[D(m,s-m-1):D(m+1,n-m-1)]_q&=[D(m+\delta_{s_1,2}\delta_{s_0,1},s-m-1):D(m+1,n-m-1)]_q&\\&+\delta_{n,m+1}\delta_{s_0,1}\delta_{s_1,2}q^{m}.\end{align*}
The proof of part (ii) follows exactly the the same ideas and is left to the reader.
Now we prove part (2) of the proposition. Note that the case $j=0$ follows from \eqref{anfan}. So by induction we can assume that  
	\begin{equation*}
[D(m,m j'+ k):D(m +1,n)]_q=
\begin{cases}
q^{j'(m j' + 2k)} &\text{ if } n=k, \\
q^{(j' + 1)(m j'+ 2k - m)} &\text{ if } n=m-k, \\
0 &\text{ otherwise. }
\end{cases}
\end{equation*}
holds for all $0\le j'<j$ , $0 \leq n, k \leq m$ such that $mj'+k\geq n$. 
If $k=0$, we can assume that $j>1$, since $j=1$ follows once more with \eqref{anfan}. 
Using the recursive formulae from part (1) and the induction hypothesis we get
	\begin{align*}
		[D(m, m j):D(m+1,n)]_q &= q^{m(2j-1)}[D(m, (j-1)m):D(m+1,n)]_q &\\&=
		\begin{cases}
q^{j(m j)} &\text{ if } n=0, \\
q^{(j + 1)(m j - m)} &\text{ if } n=m, \\
0 &\text{ otherwise. }
\end{cases}
\end{align*}
If $0 < k \leq \frac{m}{2}$, then \begin{align*}
     	[D(m,m j+k):D(m+1,n)]_q&=q^{4kj}[D(m,m j -k):D(m+1,n)]_q&\\&
     	=\begin{cases}
q^{j(m j + 2k)} &\text{ if } n=k, \\
q^{(j + 1)(m j+ 2k - m)} &\text{ if } n=m-k, \\
0 &\text{ otherwise. }
\end{cases}
\end{align*}
If $\frac{m}{2} < k \leq m$, then 
\begin{align*}
	[D(m,m j+k):D(m+1,n)]_q&=q^{(2j+1)(2k-m)}[D(m,mj+m-k):D(m+1,n)]_q&\\&
	=q^{(2j+1)(2k-m)}q^{4(m-k)j}[D(m,m(j-1)+k):D(m+1,n)]_q&\\&
	=\begin{cases}
q^{j(m j + 2k)} &\text{ if } n=k, \\
q^{(j + 1)(m j+ 2k - m)} &\text{ if } n=m-k, \\
0 &\text{ otherwise. }
\end{cases}
\end{align*}
\end{proof}
\end{prop}
\subsection{Proof of Proposition~\ref{1to2and2to3}(i)} If $s\in\{0,1\}$ the statement of the proposition follows from Proposition~\ref{hilfslem1}(2). Also the case $p=0$ is clear. Using the recursive formulae in Proposition~\ref{hilfslem1}, an easy induction argument shows
\begin{align*}[D(1,s+p):D(2,s)]_q&=q^{2p}[D(1,(s-2)+p):D(2,s-2)]_q+q^{2(s+p)-1}[D(1,s+(p-1)):D(2,s))]_q&\\&
=q^{p(s+p+\text{res}_{2}(s))}\qbinom{\lfloor \frac{s-2}{2} \rfloor+p}{p}_{q^2}+q^{2s+p}q^{(p-1)(s+p+\text{res}_{2}(s))}\qbinom{\lfloor \frac{s}{2} \rfloor+p-1}{p-1}_{q^2}&\\&
=q^{p(s+p+\text{res}_{2}(s))}\left(\qbinom{\lfloor \frac{s-2}{2} \rfloor+p}{p}_{q^2}+q^{s-\text{res}_{2}(s)}\qbinom{\lfloor \frac{s}{2} \rfloor+p-1}{p-1}_{q^2}\right)&\\&
=q^{p(s+p+\text{res}_{2}(s))}\qbinom{\lfloor \frac{s}{2} \rfloor+p}{p}_{q^2},\end{align*}
which finishes the proof.
\subsection{Proof of Proposition~\ref{1to2and2to3}(ii)} Recall the following recursive formulae from Proposition~\ref{hilfslem1}:
\begin{align*}
[D(2,2n):D(3,s)]_q&=q^{2(2n-s)}[D(2,2n-3):D(3,s-3)]_q+q^{2(2n-1)}[D(2,2n-2):D(3,s)]_q&\\&+q^{6n-2s-3}[D(2,2n-4):D(3,s-3)]_q
,\quad \text{ if $n\geq 2$,}
\end{align*}
\begin{align*}
[D(2,2n+1):D(3,s)]_q&=q^{4n+2-2s}[D(2,2n-2):D(3,s-3)]_q&\\&+q^{4n}[D(2,2n-1):D(3,s)]_q,\quad \text{ if $n\geq 1$.}
\end{align*}
Since all cases follow the same idea, we will only prove 
\begin{align}\label{firstf1}&\notag \left[ D(2,3s+r+2p) : D(3,3s+r)\right]_q =&\\& \hspace{2cm} q^{2(p^2 +p(2s+r))} \sum_{j=0}^{p}
q^{2j(j+\delta_{r,1})} \qbinom{s+p-j}{s}_{q^2}\qbinom{\lfloor\frac{s}{2}\rfloor+j-\delta_{r,1}\delta_{\text{res}_2(s),0}}{2j}_{q^2},\end{align}
where $r\in\{0,1,2\}$ and $s\in\bz_+$. The proof proceeds by induction on $s$. If $s=0$, the statement follows from Proposition~\ref{hilfslem1}(2) and the induction begins. The strategy of the proof is to show that the $q$--binomial formulae in \eqref{firstf1} satisfies the above recursive formulae. Let us define 
$$u'=\delta_{r,1}\text{res}_2(s),\quad y'=(\delta_{r,0}+\delta_{r,2})\delta_{\text{res}_2(s),0}.$$
\textit{Case 1:} In this case we suppose that $3s+r$ is even. By the induction hypothesis, it will be enough to show that the $q$--binomial formula in \eqref{firstf1} satisfies the recursive
	  \begin{align*}
          [D(2,3s+r+2p):D(3,3s+r)]_q&=q^{4p}[D(2,3(s-1)+r+2p):D(3,3(s-1)+r)]_q&\\&+q^{2(3s+r+2p-1)}[D(2,3s+r+2(p-1)):D(3,3s+r)]_q&\\&+q^{3s+r+6p-3}[D(2,3(s-1)+r+2(p-1)+1):D(3,3(s-1)+r)]_q.\end{align*}
Equivalently, it suffices to show that the following sum vanishes for all $p \geq 0$:
\begin{align*}\sum_{j=0}^p q^{2j(j+\delta_{r,1})}&\qbinom{s+p-j-1}{s-1}_{q^2}\qbinom{\lfloor \frac{s-1}{2}\rfloor+j-u'}{2j}_{q^2}&\\&
+\sum_{j=0}^{p}q^{2s+2j(j+\delta_{r,1})}\qbinom{s+p-j-1}{s}_{q^2}\qbinom{\lfloor\frac{s}{2}\rfloor+j}{2j}_{q^2}&\\&+\sum_{j=0}^{p-1}q^{2j(j+\delta_{r,1})+2j+s+\delta_{r,1}}\qbinom{s+p-j-2}{s-1}_{q^2}\qbinom{\lfloor\frac{s-1}{2}\rfloor+j+y'}{2j+1}_{q^2}&\\&-\sum_{j=0}^{p}q^{2j(j+\delta_{r,1})}\qbinom{s+p-j}{s}_{q^2}\qbinom{\lfloor\frac{s}{2}\rfloor+j}{2j}_{q^2}.\end{align*}
Note that $y'=1$ implies $u'=0$ and $s$ is even and $y'=0$ forces $u'=1$ and $s$ is odd. It follows
$$\qbinom{\lfloor \frac{s-1}{2} \rfloor+j-u'}{2j}_{q^2}+q^{s-2j-\delta_{r,1}}\qbinom{\lfloor \frac{s-1}{2}\rfloor+j+y'-1}{2j-1}_{q^2}=\qbinom{\lfloor\frac{s}{2}\rfloor+j}{2j}_{q^2},\quad 1 \leq j \leq p.$$
Now rewriting the third term in the above sum as
$$\sum_{j=1}^{p}q^{2j(j+\delta_{r,1})+s-2j-\delta_{r,1}}\qbinom{s+p-j-1}{s-1}_{q^2}\qbinom{\lfloor\frac{s-1}{2}\rfloor+j+y'-1}{2j-1}_{q^2}$$
we obtain the desired result.\\
\textit{Case 2:} We assume that $3s+r$ is odd. Again, it will be enough to show that the $q$--binomial formula in \eqref{firstf1} satisfies the recursive 
\begin{align*}[D(2,3s+r+2p):D(3,3s+r)]_q&=q^{4p}[D(2,3(s-1)+r+2p):D(3,3(s-1)+r)]_q&\\&+q^{6s+2r+4p-2}[D(2,3s+r+2p-2):D(3,3s+r)]_q,\end{align*} which is equivalent to the statement that the following sum vanishes for all $p \geq 0$:
\begin{align*}\sum_{j=0}^{p}q^{2j(j+\delta_{r,1})}\Bigg(&\qbinom{s+p-j-1}{s-1}_{q^2}\qbinom{\lfloor\frac{s-1}{2}\rfloor+j-u'}{2j}_{q^2} +q^{2s}\qbinom{s+p-j-1}{s}_{q^2}\qbinom{\lfloor\frac{s}{2}\rfloor+j-\delta_{r,1}\delta_{\text{res}_2(s),0}}{2j}_{q^2}&\\&-\qbinom{s+p-j}{s}_{q^2}\qbinom{\lfloor\frac{s}{2}\rfloor+j-\delta_{r,1}\delta_{\text{res}_2(s),0}}{2j}_{q^2}\Bigg).\end{align*}
Since $\lfloor \frac{s-1}{2}\rfloor-u'=\lfloor \frac{s}{2}\rfloor -\delta_{r,1}\delta_{\text{res}_2(s),0}$ it is easy to see that the above sum is zero. Hence Proposition~\ref{1to2and2to3}(ii) is proven.

\section{Proof of Theorem~\ref{thmgenser1}}\label{section7}
In this section, we study the generating series $A_n^{m'\rightarrow m}(x)$ for the numerical multiplicities and give a closed formula when $m'=1$ and $m=m'+1$, showing that they are rational functions. We will use freely the notation and results
established in the previous sections.
\subsection{}The following result gives a recursive formulae . It will be convinient to set $A_{-1}^{1\rightarrow m}(x)=1$.
\begin{prop}\label{genserrec}
For $m\geq 2$ and $n\geq -1$ write $n+1=m n_1+n_0$ , where $0<n_0 \leq m,\ n_1 \geq -1$. The generating series $A_n^{1 \rightarrow m}(x)$ satisfies the reccurence,
\begin{equation*}
 A_n^{1 \rightarrow m}(x)=
\begin{cases}
A_{n+1}^{1 \rightarrow m}(x)  & \text{ if } 2n_0=m-1,\\
A_{n+1}^{1 \rightarrow m}(x)(1-(1-\delta_{m,2})x)  & \text{ if } n_0 = m-1 \text{ or } 2n_0 = m-2,\\
A_{n+1}^{1 \rightarrow m}(x)-(1-\delta_{m,2})x^2 A_{n+2}^{1 \rightarrow m}(x) & \text{ if } 2n_0 = m,\\
(1-x)A_{n+1}^{1 \rightarrow m}(x)-(1-\delta_{m,2})x^2 A_{n+2}^{1 \rightarrow m}(x) & \text{ if } 2n_0 \notin \{ m-2,m-1,m \} \text{ and } n_0 \neq m-1.
\end{cases}
\end{equation*}
\end{prop}
The proof of the above proposition is postponed to the end of this section. We first discuss how this can be used to prove Theorem~\ref{thmgenser1}.
\subsection{Proof of Theorem \ref{thmgenser1}}  
Set $D_m(x)=a_m(x) a_{m+1}(x)$ and $F_k=A_k^{1 \rightarrow m}(x)$ for $k \geq -1$. The theorem follows if we prove that for all $k \geq 0$ and $0 \leq r <m$, we have 
		$$(a) F_{mk+r}=N_{m,r}(x)F_{mk+m-1},  \ \ \ \ (b) F_{mk+m-1}=\frac{1}{D_m(x)^{\lfloor\frac{k}{m}\rfloor +1}},$$
where
\begin{equation*}
N_{m,r}(x) =
\begin{cases}
a_{2m-2r-1}(x)  & \text{ if } \lfloor \frac{m}{2} \rfloor \leq r \leq m-1,\\
a_m(x) a_{m-2r-1}(x)  & \text{ if } 0 \leq r \leq \lfloor \frac{m}{2} \rfloor -1.
\end{cases}
\end{equation*}
We first prove (a). If $r=m-1$ this is immediate since $N_{m,m-1}(x)=1$ and if $r =m-2$ this follows from the second case of Proposition~\ref{genserrec}. Assume now that we have proved the equality for all $0 \leq r'<m$ with $r'> r$ and $r<m-2$. To prove the equality for $r$, we use Proposition~\ref{genserrec} and the following equalities which can be easily checked 
\begin{itemize}
\item $N_{m,\lfloor \frac{m}{2} \rfloor -1}(x)=N_{m, \lfloor \frac{m}{2} \rfloor}(x)$, if m is odd,
\item $N_{m,m-2}(x)=(1-(1-\delta_{m,2})x)N_{m,m-1}(x)$,
\item $N_{m,\frac{m}{2}-2}(x)=(1-x)N_{m,\frac{m}{2}-1}(x)$ if $m\geq 4$ is even,
\item $N_{m,\frac{m}{2}-1}(x)=N_{m,\frac{m}{2}}(x)-(1-\delta_{m,2})x^2 N_{m,\frac{m}{2}+1}(x)$ if m is even,
\item $N_{m,p-1}(x)=(1-x)N_{m,p}(x)-x^2N_{m,p+1}(x)$ if $2p \notin \{m-2,m-1,m\} $ and $1\leq p \leq m-2$,
\item $N_{m,m-1}(x)D_m(x)=(1-x)N_{m,0}(x)-(1-\delta_{m,2})x^2N_{m,1}(x)$.
\end{itemize}
Now part (a) follows by an easy induction argument. In order to prove (b), observe that the last case of Proposition~\ref{genserrec} gives $$F_{m(k-1)+m-1}=(1-x)F_{m(k-1)+m}-(1-\delta_{m,2}x^2)x^2 F_{m(k-1)+m+1}=D_m(x)F_{mk+m-1}, k\geq 0.$$ Since $F_{-1}=1$ we get $D_m(x)^{k+1}F_{mk+m-1}=1$ and the proof is complete. 
\subsection{}
The following lemma is needed in the proof of Proposition~\ref{genserrec}.
\begin{lem}\label{cruciallemma}
 Let $m\in\mathbb{N}$ and $p\in\mathbb{Z}_+$. Write $p=mp_1+p_0$ with $p_1,p_0 \in \mathbb{Z}$, $p_1 \geq -1$ and 
 $0 <p_0\leq m$. Then
\begin{align*} \ch_{\lie g} D(m,p)\ \ch_{\lie g} D(1,1)=\ch_{\lie g} D(m,p+1)&+(1-\delta_{m,1})\Big((1-\delta_{2p_0,m}-\delta_{2p_0,m+1}-\delta_{p_0,m})\ch_{\lie g} D(m,p-1)&\\&+(1-\delta_{2p_0,m}-\delta_{2p_0,m-1})\ch_{\lie g} D(m,p)\Big).\end{align*}
 \begin{proof}If $m=1$, the statement is obvious since $D(1,p)\otimes D(1,1)\cong_{\mathfrak{g}} D(1,p+1)$; also if $p=0$ there is nothing to prove. So assume from now on $m\geq 2$ and $p\neq 0$. \\
{\em{Case 1}}: Let $p_0 < m$. Then $D(m,p)\cong_{\mathfrak{g}} D(m,mp_1)\otimes D(m,p_0)$ and hence
\begin{align*}
 D(m,p) \otimes D(1,1)&\cong_{\mathfrak{g}} D(m,mp_1)\otimes D(m,p_0)\otimes D(1,1)&\\&
   \cong_{\mathfrak{g}} D(m,mp_1)\otimes D(m,p_0) \otimes \big(D(m,1)\oplus V(0)\big)&\\&\cong_{\mathfrak{g}} D(m,mp_1)\otimes \big(D(m,p_0)\otimes D(m,1)\oplus D(m,p_0)\big).
  \end{align*}
So the lemma in this case follows if we prove
\begin{align}\label{ddffrr}\ch_{\lie g}D(m,p_0)& \ \ch_{\lie g} D(m,1)=\ch_{\lie g}D(m,p_0+1)&\\&\notag+(1-\delta_{2p_0,m}-\delta_{2p_0,m+1})\ch_{\lie g} D(m,p_0-1)-(\delta_{2p_0,m}+\delta_{2p_0,m-1})\ch_{\lie g} D(m,p_0).\end{align}
We have the following decomposition into irreducible $\mathfrak{g}$--modules:
\begin{align*}D(m,p_0)\otimes D(m,1)&\cong_{\mathfrak{g}}
 \begin{cases}
  \big(V(p_0)\oplus V(p_0-1)\oplus \cdots \oplus V(m-p_0)\big)\otimes V(1),& \text{ if } 2p_0 > m,\\
  V(p_0)\otimes V(1),& \text{ if } 2p_0 \leq m,
	\end{cases}
	&\\& \cong_{\mathfrak{g}}
	 \begin{cases}
  \bigoplus^{p_0+1}_{j=m-p_0+1} V(j)\oplus \bigoplus^{p_0-1}_{j=m-p_0-1} V(j),& \text{ if } 2p_0 > m,\\
  V(p_0+1)\oplus V(p_0-1),& \text{ if } 2p_0 \leq m.
 \end{cases}
\end{align*}
Further we know
\begin{align*}
D(m,p_0+1)\cong_{\mathfrak{g}}
 \begin{cases}
  V(p_0+1)\oplus V(p_0)\oplus \cdots \oplus V(m-p_0-1),& \text{ if } 2p_0+2 > m,\\
  V(p_0+1),& \text{ if } 2p_0+2 \leq m.
 \end{cases}
\end{align*}
\begin{align*}D(m,p_0-1)\cong_{\mathfrak{g}}
 \begin{cases}
  V(p_0-1)\oplus V(p_0)\oplus \cdots \oplus V(m-p_0+1),& \text{ if } 2p_0+2 > m,\\
  V(p_0-1), &\text{ if } 2p_0+2 \leq m.
 \end{cases}
\end{align*}
Now writing both sides of \eqref{ddffrr} in terms of characters of irreducible $\lie g$--modules and using the above formulae we get \eqref{ddffrr}.\\
{\em{Case 2}}: Let $p_0=m$. We get  
$$D(m,p)\otimes D(1,1)\cong_{\mathfrak{g}} D(m,m(p_1+1))\otimes \big(D(m,1)\oplus V(0)\big)\cong_{\mathfrak{g}} D(m,p+1)\oplus D(m,p).$$
\end{proof}
\end{lem}
\subsection{Proof of Proposition~\ref{genserrec}}We continue assuming $p=mp_1+p_0$ with $p_1,p_0 \in \mathbb{Z}$, $p_1 \geq -1$ and $0 <p_0\leq m$. For $s\in\mathbb{Z}_+$, we write $$\ch_{\lie g} D(1,s)=\sum_{p \geq 0}[D(1,s):D(m,p)]_{q=1}\ch_{\lie g} D(m,p)$$
and multiply both sides with $\ch_{\lie g} D(1,1)$. We get
\begin{equation}\label{hhggvv}\ch_{\lie g} D(1,s)\ \ch_{\lie g} D(1,1)=\ch_{\lie g} D(1,s+1)=\sum_{p \geq 0}[D(1,s):D(m,p)]_{q=1}\ch_{\lie g} D(m,p)\ \ch_{\lie g}D(1,1).\end{equation}
Now we can apply Lemma~\ref{cruciallemma} to both sides of \eqref{hhggvv}. Applying to the right hand side gives a linear combination of characters of level $m$ Demazure modules. Writing the left side as
$$\ch_{\lie g} D(1,s+1)=\sum_{p \geq 0}[D(1,s+1):D(m,p)]_{q=1}\ch_{\lie g} D(m,p)$$
and equating the coefficients on both sides of \eqref{hhggvv} gives
\begin{align*}
[D(1,s+1):D(m,p)]_{q=1}&=[D(1,s):D(m,p-1)]_{q=1}+&\\&+(1-\delta_{2\widetilde{p_0},m}-\delta_{2\widetilde{p_0}.m+1}-\delta_{\widetilde{p_0},m})[D(1,s):D(m,p+1)]_{q=1}&\\&+(1-\delta_{2p_0,m}-\delta_{2p_0,m-1})[D(1,s):D(m,p)]_{q=1}
\end{align*}
where
\begin{equation*}
 \widetilde{p_0}=
 \begin{cases}
  p_0+1, &\text{ if } p_0 < m,\\
  1, &\text{ otherwise.}
 \end{cases}
\end{equation*}
Let us consider the case $2p_0=m$. We get from above 
$$
[D(1,s+1):D(m,p)]_{q=1}=[D(1,s):D(m,p-1)]_{q=1}+(1-\delta_{m,2})[D(1,s):D(m,p+1)]_{q=1}
$$
and thus
\begin{align*}A_{p-1}^{1\rightarrow m}(x)&=1+\sum_{k>0}[D(1,p-1+k):D(m,p-1)]x^k&\\&=1+\sum_{k>0}[D(1,p+k):D(m,p)]x^k-(1-\delta_{m,2})\sum_{k>0}[D(1,p-1+k):D(m,p+1)]x^k&\\&= A_p^{1\rightarrow m}(x)-(1-\delta_{m,2})x^2A_{p+1}^{1\rightarrow m}(x).
\end{align*}
Repeating the same argument for all remaining cases gives the statement of Proposition~\ref{genserrec}. We omit the details.

\subsection{}
Finally, we turn our attention to the study of $A_n^{m \rightarrow m+1}(x)$ for $m \ge 1$.
\begin{prop}\label{elltheorem}
For $m \geq 1$ and  $n \geq 0$, write $n = (m+1)p_n- r_n$  where $p_n\in\mathbb{Z}_+$ and $0 \leq r_n \leq m$. Then,
\begin{equation} \label{eq:lev_l_recc}
A_{n}^{m \rightarrow m+1}(x)=
\begin{cases}
A_{n+m}^{m \rightarrow m+1}(x) - x^{2 r_n}A_{n+2r_n}^{m \rightarrow m+1}(x) &\text{  if  } 1 \leq r_n \leq \lfloor \frac{m-1}{2} \rfloor,\\
\\
A_{n+m}^{m \rightarrow m+1}(x)- x^{2 r_n-m}A_{n+2r_n}^{m \rightarrow m+1}(x) &\text{  if  } r_n = \lfloor \frac{m+1}{2} \rfloor \text{ and } m \text{ is odd,}\\
\\
A_{n+m}^{m \rightarrow m+1}(x)- x^{2 r_n}A_{n+2r_n}^{m \rightarrow m+1}(x) &\text{  if  } r_n = \lfloor \frac{m+1}{2} \rfloor \text{ and } m \text{ is even,}\\
\\
A_{n+m}^{m \rightarrow m+1}(x) - x^{2 r_n-m}A_{n+2r_n}^{m \rightarrow m+1}(x)-x^{2r_n-m-1}A_{n+2r_n-m-1}(x) &\text{  if  } \lfloor \frac{m+3}{2} \rfloor \leq r_n \leq  m, \\
\\
A_{n+m}^{m \rightarrow m+1}(x)   &\text{ if } m+1 \mid n. 
\end{cases}
\end{equation}
\end{prop}
\begin{proof}
To simplify notation, we fix $m \geq 1$ and for $s,n \in \mathbb{Z}_+$, set
$$\nu(s,n)=[D(m,s):D(m+1,n)]_{q=1}.$$
Recall that $\nu(s,n)=0$ if $s<n$. The theorem follows if we prove for all $s,n\geq 0$ that
\begin{equation}\label{eq: nmultassert}
\nu(s,n)=\begin{cases}
\nu(s+m,n+m)-\nu(s,n+2r_n) &\text{ if } 1 \leq r_n \leq \lfloor \frac{m-1}{2}\rfloor,\\
\nu(s+m,n+m)-\nu(s+m,n+2r_n) &\text{ if } r_n=\lfloor\frac{m+1}{2}\rfloor \text{ and } m \text{ is odd},\\
\nu(s+m,n+m)-\nu(s,n+2r_n) &\text{ if } r_n=\lfloor\frac{m+1}{2}\rfloor \text{ and } m \text{ is even},\\
\nu(s+m,n+m)-\nu(s+m,n+2r_n)-\nu(s,n+2r_n-m-1) &\text{ if } \lfloor \frac{m+3}{2}\rfloor\leq r_n\leq m,\\
\nu(s+m,n+m) &\text{ if } m+1\mid n.
\end{cases}
\end{equation}
Notice that this equality holds whenever $s<n$, since both sides are zero. Hence we can assume that $s \geq n$. Taking $q=1$ in Proposition~\ref{hilfslem1} gives 
\begin{equation}\label{eq: facttwo}
s \ge 0, \; 0\le n\le m\implies
\nu(s,n) = \begin{cases} 1 & \text{if } s+n \text{ or } s-n \text{ is a multiple of } m,\\
0 & \text{ otherwise}.
\end{cases}
\end{equation}
Applying once more Proposition~\ref{hilfslem1} with $s=m s_1+s_0$, $0<s_0\leq m$, $s_1\geq 1$ and $\bxi=((m+1),m^{s_1},s_0)$ gives 
\begin{align*}
\nu(s,n)=\nu(s-m-1,n-m-1)&+\nu(s+m-2s_0,n)&\\&+\sum_{j=1}^{k(\bxi)}\nu(s-2s_0+j-1,n-m-1),\quad \text{if $2s_0>m$}
\end{align*}
and
\begin{gather*}
\nu(s,n)=\nu(s-m-1,n-m-1)+\nu(s-2s_0,n),\quad \text{if $2s_0\leq m$}.
\end{gather*}
We now proceed to prove equation \eqref{eq: nmultassert} by induction on $n$. Equation \eqref{eq: facttwo} obviously implies $\nu(s,0)=\nu(s+m,m)$ and hence induction begins. Now let $n>0$. Assume that $\nu(s,n')$ satisfies \eqref{eq: nmultassert} for all $0\le n'<n$ and for all $s\in\mathbb{Z}_+$. We proceed by induction on $s$ to prove that $\nu(s,n)$  satisfies \eqref{eq: nmultassert} for all $s\in\mathbb{Z}_+$. Notice that the $s=0$ is clear since both sides of  \eqref{eq: nmultassert} are zero. Further, as remarked earlier, we assume for the rest of the proof that $s \ge n$. \\
{\em Case 1:} Suppose $0<n\le m$ and $1 \leq r_n \leq \lfloor \frac{m-1}{2}\rfloor$. We have to prove that
$$\nu(s,n)= \nu(s+m,n+m)-\nu(s,2m+2-n).$$

{\em Case 1(a):} Suppose $s \geq m+1$ and $2s_0>m$. Then the recursive can be used for both terms of the right hand side and we get
$$
\nu(s+m,n+m)=\nu(s-1,n-1)+\nu(s+2m-2s_0,n+m)+\sum_{j=1}^{k(\bxi)}\nu(s+m-2s_0+j-1,n-1),
$$
\begin{align*}
\nu(s,2m+2-n)=\nu(s-m-1,m+1-n)&+\nu(s+m-2s_0,2m+2-n)&\\&+\sum_{j=1}^{k(\bxi)}\nu(s-2s_0+j-1,m+1-n).
\end{align*}
Set
$$T_1 = \nu(s-1,n-1) - \nu(s-m-1,m+1-n),$$  
$$T_2 =  \nu(s+2m-2s_0,n+m) - \nu(s+m-2s_0,2m+2-n),$$ and $$T_3=\nu(s+m-2s_0+j-1,n-1)-\nu(s-2s_0+j-1,m+1-n).$$

Equation \eqref{eq: facttwo} applies to both terms in $T_1$ and $T_3$. Since
\begin{align*}
(s-1) - (n-1) &= (s-m-1) + (m+1-n) \\
(s-1) + (n-1) &\equiv (s-m-1) - (m+1-n) \pmod{m},
\end{align*}
\begin{align*}
(s+m-2s_0+j-1) - (n-1) &= (s-2s_0+j-1) + (m+1-n) \\
(s+m-2s_0+j-1) + (n-1) &\equiv (s-2s_0+j-1) - (m+1-n) \pmod{m},
\end{align*}
we deduce that $T_1 =T_3=0$. Further, since $s-2s_0 < s$, the inductive hypothesis gives $T_2 = \nu(s+m-2s_0,n)$.  Hence we have to verify  $\nu(s,n) = \nu(s+m-2s_0,n)$. Since $s \equiv s_0 \pmod{m}$, we obtain $s +m- 2s_0 \equiv -s \pmod{m}$ and thus $s \pm n \equiv -(s+m-2s_0 \mp n) \pmod{m}$. Applying \eqref{eq: facttwo} completes the proof.\\
{\em Case 1(b):}Suppose that $s \geq m+1$ and $2s_0 \leq m$. Then the recursive formulae can be used for both terms on the right hand side and we get 
$$\nu(s+m,n+m)=\nu(s-1,n-1)+\nu(s-2s_0+m,n+m),$$
$$\nu(s,2m+2-n)=\nu(s-m-1,m+1-n)+\nu(s-2s_0,2m+2-n).$$ Since $T_1=0$ and $s-2s_0 < s$, by using induction on $s$, we get $$\nu(s-2s_0+m,n+m)-\nu(s-2s_0,2m+2-n)=\nu(s,n)$$ which proves the claim.\\
{\em Case 1(c):} Suppose $s \leq m$. Since $2m + 2 -n > m$, we have
$\nu(s,2m+2-n) =0$. Thus we need to show that
$\nu(s,n)= \nu(s+m,n+m)$. Applying the recursive formulae again we get
$$\nu(s+m,n+m)= \nu(s-1,n-1)+\nu(s+2m-2s_0,n+m).$$
But since $0 <s \leq m$, we have $s =s_0$, and hence $s+2m-2s_0=2m-s_0  < n+m$. Thus the second term vanishes. It remains to show $\nu(s-1,n-1) = \nu(s,n)$. But this is clear by \eqref{eq: facttwo}, which implies for $1 \leq s,n \leq m$: $\nu(s-1,n-1) = \nu(s,n) = \delta_{s,n}$.\\
{\em Case 2:} Suppose $n \ge m+1$ and $1 \leq r_n \leq \lfloor \frac{m-1}{2}\rfloor$. Then $m + 1 \nmid r$. Consider
$$S = \nu(s+m,n+m) - \nu(s,n+2r_n) - \nu(s,n).$$ 
{\em Case2(a):}
Assume that $2s_0 > m$. By applying the recursive to each of the terms of $S$,
we have
\begin{align*}S &=\nu(s-1,n-1)+\nu(s+2m-2s_0,n+m)+\sum_{j=1}^{k(\bxi)}\nu(s+m-2s_0+j-1,n-1)\\&-\nu(s-m-1,n+2r_n-m-1)-\nu(s+m-2s_0,n+2r_n)-\sum_{j=1}^{k(\bxi)}\nu(s-2s_0+j-1,n+2r_n-m-1)\\ &-\nu(s-m-1,n-m-1)-\nu(s+m-2s_0,n)-\sum_{j=1}^{k(\bxi)}\nu(s-2s_0+j-1,n-m-1).\end{align*}
Since $n-m-1<n$, $s+m-2s_0<s$ and $s-2s_0+j-1<s$, the inductive hypothesis gives
\begin{align*}
\nu(s-m-1,n-m-1)&=\nu(s-1,n-1)-\nu(s-m-1,n+2r_n-m-1),\\
\nu(s+m-2s_0,n)&=\nu(s+2m-2s_0,n+m)-\nu(s+m-2s_0,n+2r_n),\\
\nu(s-2s_0+j-1,n-m-1)&=\nu(s+m-2s_0+j-1,n-1)-\nu(s-2s_0+j-1,n+2r_n-m-1).
\end{align*}
Substituting these equations in our equation for $S$, we obtain $S=0$.\\
{\em Case2(b):} Let us assume that $2s_0 \leq m$. By applying the recursive to each of the terms of $S$,
\begin{align*}
S=\nu(s-1,n-1)+\nu(s+m-2s_0,n+m)-\nu(s-m-1,n+2r_n-m-1)\\-\nu(s-2s_0,n+2r_n)-\nu(s-m-1,n-m-1)-\nu(s-2s_0,n).
\end{align*}
Again by using the induction hypothesis we get
\begin{align*}
\nu(s-m-1,n-m-1)=\nu(s-1,n-1)-\nu(s-m-1,n+2r_n-m-1),\\
\nu(s-2s_0,n)=\nu(s+m-2s_0,n+m)-\nu(s-2s_0,n+2r_n).
\end{align*}
Substituting these equations in our equation for $S$, we obtain $S=0$.
 The remaining cases can be proven similarly and we leave the details to the reader. 
\end{proof}

\subsection{} We shall use Proposition~\ref{elltheorem} to establish a closed formula for the generating series $A_{n}^{m \rightarrow m+1}(x)$. For this we define polynomials $d_n(x)$, $n\ge 0$  with non--negative integer coefficients as follows. For $0 \leq n \leq m$ set \begin{equation*}
 d_n(x)=\begin{cases}
  1 &\text{ if } n=0,m, \\
  1+x^{m-2n} &\text{ if } 1<n\leq \lceil\frac{m-2}{2}\rceil, \\
  1+\delta_{res_2(m),1}x &\text{ if } n=\lceil\frac{m-1}{2}\rceil, \\
  1+x^{2m-2n} &\text{ if } \lceil\frac{m}{2}\rceil \leq n \leq m-1.
  \end{cases}
\end{equation*}

The polynomials $d_n(x)$ for $n>m$ are defined by requiring that the following equality holds for all $p\ge 1$: \begin{equation}\label{definitionofd} \begin{bmatrix} d_{(m+1)p}(x) & d_{(m+1)p+1}(x) & \cdots &
  d_{(m+1)p+m}(x)\end{bmatrix}^{T}= K^{p} \begin{bmatrix} d_0(x) & d_1(x) & \cdots &
  d_m(x) \end{bmatrix}^T,\end{equation}
where $K$ is a $(m+1)\times(m+1)$ matrix defined as:
$$K=
\begin{cases} 
  K_1+K_{2,0} &\text{    if    } m \text{ is even }\\
  K_1+K_{2,1} &\text{    if    } m \text{ is odd  }
 \end{cases}
$$
 and the three $(m+1)\times(m+1)$ matrices $K_1, K_{2,0} \text{ and } K_{2,1}$ are defined below. Set
$$K_1=
\begin{bmatrix}
0  & 1     &       &          &          &           &           &              \\        
  & 0     & 1     &          &          &           &           &              \\        
   &       &\ddots & \ddots   &          &           &           &              \\       
   &       &       & 0        & 1        &           &           &              \\      
   &       &       &          &x^{m-1}&1          &           &              \\      
  
   &       &       &          &          &\ddots     &\ddots     &              \\   
   &       &       &          &          &           &  x^{m-1}         &1              \\  
   &       &       &          &          &           &           & x^{m-1}        
\end{bmatrix}
$$
where number of zeros on the diagonal of the matrix $K_1$ is $\lceil \frac{m+1}{2}\rceil$ and the number of entries equal to $x^{m-1}$ is $\lfloor \frac{m+1}{2}\rfloor$. Set
$$K_{2,0}= 
\begin{bmatrix}
 &     &         &        &          &          &          &          &          &            &x^{m-1}  \\ 
 &     &         &        &          &          &          &          &          &x^{m-3}  &x^{m-2}  \\ 
 &     &         &        &          &          &          &          &\iddots   &x^{m-4}  &            \\  
 &     &         &        &          &          &          &x         &\iddots   &            &            \\  
 &     &         &        &          &          &0         &x^2       &          &            &            \\   
 &     &         &        &          &0         &0         &          &          &            &            \\   
 &     &         &        &\iddots   &x^{m-2}&          &          &          &            &            \\    
 &     &         &\iddots &x^{m-4}&          &          &          &          &            &            \\     
 &     &0        &\iddots &          &          &          &          &          &            &            \\     
&0    &x^2      &        &          &          &          &          &          &            &            \\     
0&1    &         &        &          &          &          &          &          &            & 

\end{bmatrix}$$
where the number of antigonal elements equal to zero in the matrix $K_{2,0}$ is equal to $\lceil \frac{m+1}{2}\rceil$. Set
$$K_{2,1}= 
\begin{bmatrix}
 &     &         &        &          &          &          &          &          &            &x^{m-1}  \\ 
 &     &         &        &          &          &          &          &          &x^{m-3}  &x^{m-2}  \\ 
 &     &         &        &          &          &          &          &\iddots   &x^{m-4}  &            \\  
 &     &         &        &          &          &          &x^2         &\iddots   &            &            \\  
 &     &         &        &          &          &0         &x       &          &            &            \\   
 &     &         &        &          &0         &0         &          &          &            &            \\   
 &     &         &        &\iddots   &x^{m-3}&          &          &          &            &            \\    
 &     &         &\iddots &x^{m-5}&          &          &          &          &            &            \\     
 &     &0        &\iddots &          &          &          &          &          &            &            \\     
 &0    &x^2      &        &          &          &          &          &          &            &            \\     
0&1    &         &        &          &          &          &          &          &            & 

\end{bmatrix}
$$
where the number of antigonal elements equal to zero in the matrix $K_{2,1}$ is equal to $\lceil \frac{m+3}{2}\rceil$.

\begin{thm}\label{closedform} Let $m \geq 1$. Then, for all $n \geq 0$, we have
  $$A_{n}^{m \rightarrow m+1}(x)=\frac{d_{n}(x)}{(1-x^{m})^{\left\lfloor \frac{n}{m+1} \right\rfloor +1}}.$$
\end{thm}
\begin{proof}
Let $n,p_n\in\mathbb{Z}_+$, $0\leq r_n\leq m$ such that $n=(m+1)p_n-r_n$. To simplify the notation in the proof we set $A_{n}:=A_{n}^{m \rightarrow m+1}(x)$ and
$$e_n(x)=A_{n}(1-x^{m})^{\left\lfloor \frac{n}{m+1} \right\rfloor +1}.$$
The strategy of the proof is to show by induction on $p$ that 
\begin{equation}\label{definitionofd2} \begin{bmatrix} e_{(m+1)p}(x) & e_{(m+1)p+1}(x) & \cdots &
  e_{(m+1)p+m}(x)\end{bmatrix}^{T}= K^{p} \begin{bmatrix} d_0(x) & d_1(x) & \cdots &
  d_m(x) \end{bmatrix}^T.\end{equation}
 From Proposition~\ref{hilfslem1}, it is easy to see that 
	$$A_{n}=\frac{d_{n}(x)}{(1-x^{m})^{\left\lfloor \frac{n}{m+1} \right\rfloor +1}}  \ \ \text{ for } 0 \leq n \leq m$$
and the induction begins. First we verify certain recursive formulae for the $e_n(x)$'s using Theorem~\ref{elltheorem}.\\
\textit{Case 1:} If $1 \leq r_n \leq \lfloor \frac{m-1}{2} \rfloor$, we have
	\begin{equation}\label{dfirst}
	A_{(m+1)p_n-r_n}=A_{(m+1)p_n+m-r_n} - x^{2 r_n}A_{(m+1)p_n+r_n},	
	\end{equation}
\begin{equation}\label{dsecond}
         A_{(m+1)p_n-(m-r_n)}=A_{(m+1)p_n+r_n}-x^{m-2r_n}A_{(m+1)p_n+(m-r_n)}-x^{m-2r_n-1}A_{(m+1)p_n-r_n-1}.	
         \end{equation}
Multiplying equation \eqref{dfirst} by $x^{m-2r_n}$ and adding the result to equation \eqref{dsecond} gives
\begin{gather*}
	(1-x^m)A_{(m+1)p_n+r_n}=A_{(m+1)p_n-(m-r_n)}+x^{m-2r_n}A_{(m+1)p_n-r_n}+x^{m-2r_n-1}A_{(m+1)p_n-r_n-1}.
\end{gather*}
It follows
\begin{equation}\label{recursiond1}
	e_{(m+1)p_n+r_n}(x)=e_{(m+1)p_n-(m-r_n)}(x)+x^{m-2r_n}e_{(m+1)p_n-r_n}(x)+x^{m-2r_n-1}e_{(m+1)p_n-r_n-1}(x).
\end{equation}
\textit{Case 2:} If $\lfloor \frac{m+3}{2} \rfloor \leq r_n \leq m-1$ we obtain
\begin{equation}\label{dthird}
	A_{(m+1)p_n-r_n}=A_{(m+1)p_n+m-r_n}-x^{2r_n-m}A_{(m+1)p_n+r_n}-x^{2r_n-m-1}A_{(m+1)p_n+r_n-m-1},
\end{equation}
\begin{equation}\label{dfourth}
	A_{(m+1)p_n-(m-r_n)}=A_{(m+1)p_n+r_n}-x^{2m-2r_n}A_{(m+1)p_n+m-r_n}.
\end{equation}
Multiplying equation \eqref{dthird} by $x^{2m-2r_n}$ and adding the result to equation \eqref{dfourth} gives
\begin{gather*}
	(1-x^m)A_{(m+1)p_n+r_n}=A_{(m+1)p_n-(m-r_n)}+x^{2m-2r_n}A_{(m+1)p_n-r_n}^{m \rightarrow m+1}(x)+x^{m-1}A_{(m+1)p_n+r_n-m-1}.
\end{gather*}
This implies
\begin{equation}\label{recursiond2}
	e_{(m+1)p_n+r_n}(x)=e_{(m+1)p_n-(m-r_n)}(x)+x^{2m-2r_n}e_{(m+1)p_n-r_n}(x)+x^{m-1}e_{(m+1)(p_n-1)+r_n}(x).
\end{equation}
\textit{Case 3:} If $r_n=m$, then 
\begin{align*}
	A_{(m+1)p_n-m}&=A_{(m+1)p_n}-x^m A_{(m+1)p_n+m}-x^{m-1}A_{(m+1)p_n-1}&\\&
	=(1-x^m)A_{(m+1)p_n+m}-x^{m-1}A_{(m+1)p_n-1}.
\end{align*}
Hence
\begin{equation}\label{recursiond3}
	e_{(m+1)p_n+m}(x)=e_{(m+1)(p_n-1)+1}(x)+x^{m-1}e_{(m+1)(p_n-1)+m}(x).
\end{equation}
\textit{Case 4:} If $m$ is even and $r_n=\frac{m}{2}$, we get
\begin{align*}
	A_{(m+1)p_n-\frac{m}{2}}&=A_{(m+1)p_n+\frac{m}{2}}-x^m A_{(m+1)p_n+\frac{m}{2}}&\\&
	=(1-x^m)A_{(m+1)p_n+\frac{m}{2}}.
\end{align*}
This shows
\begin{equation}\label{recursiond4}
	e_{(m+1)p_n+\frac{m}{2}}(x)=e_{(m+1)(p_n-1)+\frac{m}{2}+1}(x). 
\end{equation}
\textit{Case 4:} If $m$ is odd, then we get the following two equations by setting $r_n=\lfloor \frac{m+1}{2} \rfloor$ and $r_n=\frac{m-1}{2}$ respectively:
\begin{equation}\label{dfifth}
	A_{(m+1)p_n-\frac{m+1}{2}}=A_{(m+1)p_n+\frac{m-1}{2}}-xA_{(m+1)p_n+\frac{m+1}{2}}^{m \rightarrow m+1}(x),
\end{equation}
\begin{equation}\label{dsixth}
	A_{(m+1)p_n-\frac{m-1}{2}}=A_{(m+1)p_n+\frac{m+1}{2}}-x^{m-1}A_{(m+1)p_n+\frac{m-1}{2}}.
\end{equation}
Multiplying equation \eqref{dsixth} by $x$ and adding the result to equation \eqref{dfifth} yields
\begin{gather*}
	(1-x^m)A_{(m+1)p_n+\frac{m-1}{2}}=A_{(m+1)p_n-\frac{m+1}{2}}^{m \rightarrow m+1}(x)+xA_{(m+1)p_n-\frac{m-1}{2}}.
\end{gather*}
Thus
\begin{equation}\label{recursiond5}
e_{(m+1)p_n+\frac{m-1}{2}}(x)=e_{(m+1)p_n-\frac{m+1}{2}}(x)+xe_{(m+1)p_n-\frac{m-1}{2}}(x).
\end{equation}
Similarly multiplying equation \eqref{dfifth} by $x^m$ and adding the result to equation \eqref{dsixth} gives
\begin{gather*}
	(1-x^m)A_{(m+1)p_n+\frac{m+1}{2}}=A_{(m+1)p_n-\frac{m-1}{2}}^{m \rightarrow m+1}(x)+x^{m-1}A_{(m+1)p_n-\frac{m+1}{2}},
\end{gather*}
which implies 
\begin{equation}\label{recursiond6}
	e_{(m+1)p_n+\frac{m+1}{2}}(x)=e_{(m+1)(p_n-1)+\frac{m+1}{2}+1}(x)+x^{m-1}e_{(m+1)(p_n-1)+\frac{m+1}{2}}(x).
\end{equation}
Hence the recursive formulae for the $e_n(x)$'s from equations \eqref{recursiond1},  \eqref{recursiond2}, \eqref{recursiond3}, \eqref{recursiond4}, \eqref{recursiond5} and \eqref{recursiond6} imply for $p>0$

\begin{align*}K\begin{bmatrix} e_{(m+1)(p-1)}(x) & e_{(m+1)(p-1)+1}(x) & \cdots &
  e_{(m+1)(p-1)+m}(x)\end{bmatrix}^{T}\\=\begin{bmatrix} e_{(m+1)p}(x) & e_{(m+1)p+1}(x) & \cdots & e_{(m+1)p+m}(x)\end{bmatrix}^{T}.
\end{align*}
Applying the induction hypothesis shows \eqref{definitionofd2}.
\end{proof}


\bibliographystyle{plain}
\bibliography{bibfile}

\end{document}